\documentclass[a4paper,12pt]{revtex4}

\usepackage[utf8]{inputenc}
\usepackage{amsmath,amsfonts,amssymb,amsthm}
\usepackage{graphicx}
\usepackage[usenames]{color}
\usepackage{enumerate}
\usepackage{float}
\usepackage{hyperref}
\usepackage{comment}
\usepackage{mathrsfs}
\usepackage{braket}
\usepackage{lscape}
\usepackage{array}
\usepackage{cleveref}
\usepackage{mathtools}
\allowdisplaybreaks

\makeatletter
\renewcommand{\maketag@@@}[1]{\hbox{\m@th\normalsize\normalfont#1}}%
\makeatother
\DeclareMathOperator*{\sbigotimes}{\text{\raisebox{0.25ex}{\scalebox{0.8}{$\bigotimes$}}}}

\allowdisplaybreaks[1]

\newtheorem{prop}{PROPOSITION}

\newtheorem{corol}{COROLLARY}

\newtheorem{conje}{CONJECTURE}

\newcommand{\eqd}{\vcentcolon=}
\newcommand{\be}{\begin{equation}}
\newcommand{\ee}{\end{equation}}
\newcommand{\bs}{\begin{split}}
\newcommand{\es}{\end{split}}
\newcommand{\ba}{\begin{array}}
\newcommand{\ea}{\end{array}}
\newcolumntype{C}[1]{>{\centering\let\newline\\\arraybackslash\hspace{0pt}}m{#1}}
\newcommand{\rn}[1]{\textup{\uppercase\expandafter{\romannumeral#1}}}

\begin{document}
%% Classifying quantum entanglement through orthogonal arrays
\title{Coarse-grained entanglement classification through orthogonal arrays}
\author{Luigi Seveso}
\affiliation{Quantum Technology Lab, Dipartimento di Fisica, Universit\`{a} degli Studi di Milano, I-20133 Milano, Italy}
\author{Dardo Goyeneche}
\affiliation{Institute of Physics, Jagiellonian University, 30-348 Krak\'ow, Poland}
\affiliation{Faculty of Applied Physics and Mathematics, Gda\'{n}sk University of Technology, 80-233 Gda\'{n}sk, Poland}
\affiliation{Departamento de F\'{i}sica, Facultad de Ciencias B\'{a}sicas, Universidad de Antofagasta, Casilla 170, Antofagasta, Chile}
\author{Karol {\.Z}yczkowski}
\affiliation{Institute of Physics, Jagiellonian University, 30-348 Krak\'ow, Poland}
\affiliation{Center for Theoretical Physics, Polish Academy of Sciences, 02-668 Warsaw, Poland}
\date{June 27, 2018}

\begin{abstract}
Classification of entanglement in multipartite quantum systems is an open problem solved so far only for bipartite systems and for systems composed of three and four qubits. We propose here a coarse-grained classification of entanglement in systems consisting of $N$ subsystems with an arbitrary number of internal levels each, based on properties of orthogonal arrays with $N$ columns.
In particular, we investigate in detail a subset of highly entangled pure states which contains all states defining maximum distance separable codes. To illustrate the methods presented, we analyze systems of four and five qubits, as well as heterogeneous tripartite systems consisting of two qubits and one qutrit or one qubit and two qutrits. 
\end{abstract}

\maketitle

\section{Introduction}\label{S1}
Characterization of entanglement in multipartite systems has proven particularly challenging, even for pure states. As one moves from a bipartite to a multipartite scenario involving $N$ parties, the algebraic structure becomes much richer. A pure multipartite quantum state is described by a tensor with $N$ indices, for which there exists no exact analogue of the singular value decomposition of a matrix which, for bipartite systems, leads to the Schmidt decomposition \cite{Pe95}. Nonetheless, study of multipartite entanglement remains a crucial goal, with applications ranging from quantum information processing, quantum computation and quantum metrology to condensed matter and many-body physics \cite{HHHH09,amico2008entanglement}. 
% wen book

One notable problem is to characterize all the different ways in which a multipartite system can share entanglement among its subsystems. A natural approach is to consider two states equivalent, if it is possible to obtain one from the other via local operations assisted by classical communication, with a nonzero probability \cite{bennett2000exact}. Protocols based on stochastic local operations and classical communication (SLOCC) correspond formally to the action of the group of invertible local transformations 
$\text{GL($d$)}^{\otimes N}$, where $d$ is the dimension of the local Hilbert space $\mathcal H_{d}$ of each of the $N$ subsystems \cite{dur2000three}. The  problem of entanglement classification is thus equivalent to finding the orbits of $\mathcal H_{d}^{\otimes N}$ under the group $\text{GL($d$)}^{\otimes N}$. 
A few results are available in this regard. For instance, for three qubits there exist six entanglement classes \cite{dur2000three}, of which only two display genuine tripartite entanglement: the GHZ class, containing states equivalent to the GHZ state $\ket{GHZ_3}=\ket{000}+\ket{111}$, and the W class, generated by the state $\ket{W_3}=\ket{001}+\ket{010}+\ket{100}$. However, already for four qubits, there is an infinite number of SLOCC classes, which can be naturally organized into nine continuous families \cite{VDMV02}. The situation gets more complicated for larger numbers of qubits. Some alternative approaches, based on topological ~\cite{kauffman2002quantum,QA18} or algebraic ~\cite{buniy2012algebraic} techniques, provide only first steps towards the resolution of the problem.

In this paper, we analyze various classes of multipartite entanglement in a simplified setting, by restricting the Hilbert space to a discrete set. Each state in the set is related to a certain combinatorial design, called an \emph{orthogonal array} (OA). An orthogonal array is a rectangular array of symbols taken from an alphabet of $d$ letters, whose combinatorial properties are characterized by its \emph{strength} -- defined in the next Section.

Orthogonal arrays were introduced by Rao in 1947 \cite{rao1947factorial} and find several applications ranging from cryptography and coding theory to the statistical design of experiments, software testing and quality control \cite{GOA05}. Moreover, they generalize some remarkable classes of combinatorial designs: Graeco-Latin squares, Hadamard matrices and classical codes \cite{hedayat1999orthogonal}. Useful libraries of these arrays can be found in the handbook on Combinatorial Designs by Colbourn and Dinitz \cite{colbourn2006handbook}, as well as in the online catalogs provided by Sloane \cite{sloaneweb} and Kuhfeld \cite{kuhfeldweb}.

Any orthogonal array with $N$ columns generates a multipartite pure state $|\psi\rangle$ of a quantum system made up of $N$ parties. In general, the pure states of a quantum system form a continuous set, but only a discrete subset of states, referred to as \emph{array-based}, is associated to an orthogonal array. In previous work \cite{goyeneche2014genuinely}, a particular class of orthogonal arrays of strength $k$, called irredundant, was shown to generate a class of quantum states, called $k$-uniform, defined by the property that each state is maximally entangled with respect to any splitting of the $N$ subsystems into two subsets composed of $k$ and $N-k$ subsystems. Here we will extend this observation much further, by studying properties of quantum states associated with different classes of orthogonal arrays. 

The aim of this paper is to approach the problem of entanglement classification for a general system containing $N$ subsystems with $d$ levels each, by making use of results from the theory of orthogonal arrays with $N$ columns and symbols from a set of $d$ letters. As a byproduct, we construct a particular family of pure states, related to \emph{generating orthogonal arrays} -- see Section \ref{S3}. These quantum states exhibit a high degree of entanglement, making them potentially interesting for various quantum information processing tasks, such as quantum teleportation in a multi-user setting, quantum key distribution and quantum error-correcting codes.

The paper is organized as follows. In Section \ref{S2} we review basic properties of orthogonal arrays and in Section \ref{S3} we show how to derive systematically the generating arrays. In Section \ref{S4}, the problem of entanglement classification for the set of array-based quantum states is formulated. In Section \ref{S5} we show how the theory of quantum entanglement can be used to determine whether two given arrays are non-isomorphic. In Section \ref{S6}, we describe the classes of quantum entanglement for states corresponding to generating arrays. The main results of this work are summarized in Section \ref{S7}. A complete list of generating OAs for systems of two up to four qubits are presented in Appendix \ref{app1}. Similar data for  systems consisting of five qubits and for heterogeneous tripartite systems consisting of qubits and qutrits are available online.

\section{Orthogonal arrays: definition and basic properties}\label{S2}
An \emph{orthogonal array} OA($r,N,d,k$) is a rectangular arrangement with $N$ columns and $r$ rows containing symbols taken from the alphabet $\mathscr A_d = \{0,\,1,\dots,\,d-1\}$, such that for any subarray made up of $k$ columns, every possible combination of $k$ symbols is repeated the same number of times $\lambda$. The parameters of the array are usually called the number of \emph{runs} $r$,
the number of \emph{factors} $N$, the \emph{index} $\lambda$ and the \emph{strength} $k$ of the OA, respectively \cite{hedayat1999orthogonal}. 

This terminology is explained by the connection with experimental design theory. Every column of an orthogonal array represents one factor, each with $d$ possible values or levels; every row represents a different run or treatment combination. By definition, the \emph{full factorial design}, denoted by $\mathscr F_{N,d}$, is the array made up of all $d^N$ treatment combinations. Unless the number of factors is small, it becomes exponentially costly to perform a full factorial experiment. One thus resorts to a \emph{fractional factorial design}, where only a subset of all treatment combinations is allowed as possible runs. An orthogonal array with $r < d^N$ is a fractional factorial design, with the strength $k$ being related to the estimation properties of different factorial effects -- for more information consult Ref.~\cite{hedayat1999orthogonal}.

In this paper, we consider only OAs with number of runs not greater than $d^N$. We also do not distinguish between arrays differing only by a reordering of their runs; in fact, as factorial designs, they are equivalent for all statistical purposes. Moreover, if two arrays can be obtained one from the other by implementing any composition of permutations of rows, columns and symbols within columns, then the arrays are called \emph{isomorphic}. 

While the strength $k$ of an array is a simple measure of its usefulness for statistical applications, there is a variety of arrays with the same strength. To discriminate among them, different quality factors have been developed in the literature \cite{dean2015handbook}. In particular, the \emph{generalized resolution} (GR) \cite{deng1999generalized,tang1999minimum,tang2001theory,gromping2014generalized} 
is a commonly used quality factor; it has a clear statistical interpretation for two-level OAs \cite{deng2002design}. In order to compute the GR, each entry $a_{ij}$ of the OA is encoded as a complex root of unity, by mapping each symbol $s\in\mathscr A_d$  to $\omega_s=\text{exp}(2\pi i s /d)$. Given the multi-index $I = j_1 \dots j_n$, denoting a subset of $n<N$ columns, the corresponding $J$-characteristic of order $n$ is defined as 
\be
J_n(I) \eqd \left|\sum_{i=1}^r \omega_{a_{ij_1}}\cdot \omega_{a_{ij_2}}\dots \omega_{a_{ij_n}}\right|=\left|\sum_{i=1}^r e^{\frac{2\pi i}{d}\sum_{j\in I} a_{ij} }\right|\;,
\ee
 where $|z|$ is the absolute value of $z\in \mathbb C$. Let $t$ be the smallest integer such that there exists at least one $J$-characteristic of order $t$ different from zero (let us remark that $t>k$ since all possible $J$-characteristics of order up to $k$ are zero). Define $J_t^{(\text{max})} \eqd \underset{I: |I|=t}{\text{max}} J_t(I)$, where $|I|$ denotes the cardinality of the multi-index $I$. Then, the generalized resolution is defined as
\be\label{GRdef}
GR \eqd t + 1 - \frac{J_t^{(\text{max})}}{r}\;.
\ee
Clearly, $t<GR<t+1$. For two-level OAs, the above quantity is invariant under isomorphisms. However, for multi-level arrays, permutations of symbols within columns can change quantity {\ref{GRdef}) \cite{cheng2004geometric}. Hence the generalized resolution of a multi-level OA is defined as the maximum value of \eqref{GRdef} taken over the family of all isomorphic arrays. 

Let the set of all orthogonal arrays OA($r,N,d,k$) with fixed values of parameters $N$, $d$ and $k$, but an arbitrary number of runs $r\leq d^N$, be denoted by $\mathscr{O\!A}(N,d,k)$. Characterizing this set for general values of its parameters is a challenging problem and only partial results are  known \cite{clark2001equivalence,angelopoulos2007effective,schoen2010complete}. For instance, complete catalogs of non-isomorphic orthogonal arrays, denoted by OA($12,N,2,k$), OA($16,N,2,k$) and OA($20,N,2,k$), for any number of rows $N$ and strength $k$, are provided in Ref.~\cite{sun2008algorithmic}. Furthermore, OA($18, N,3,k$) were classified \cite{tsai2006complete} for all possible $N$ and $k$. A method to enumerate non-isomorphic  two-level orthogonal arrays of strength $k$ and index $\lambda$ with $N=k+2$ columns and $r = 2k\lambda$ rows was presented in ~\cite{stufken2007complete}. Further orthogonal arrays with a small number of runs have been classified in \cite{Ye2008construction}.

\section{Generating orthogonal arrays}\label{S3}
In this section, we describe a systematic method which, in principle, allows one 
to obtain a complete description of the set $\mathscr{O\!A}(N,d,k)$
for general values of the parameters.
In practice, computational complexity limits its applicability to low-dimensional cases. 

Let us introduce the composition function, denoted by $\oplus$, 
and defined as follows,
\be
\begin{array}{cccc}
a_{1,1}&a_{1,2}&\dots&a_{1,N}\\
a_{2,1}&a_{2,2}&\dots&a_{2,N}\\
\vdots&\vdots&&\vdots\\
a_{r,1}&a_{r,2}&\dots&a_{r,N}\\
\end{array} \ \oplus \ \begin{array}{cccc}
a'_{1,1}&a'_{1,2}&\dots&a'_{1,N}\\
a'_{2,1}&a'_{2,2}&\dots&a'_{2,N}\\
\vdots&\vdots&&\vdots\\
a'_{s,1}&a'_{s,2}&\dots&a'_{s,N}\\
\end{array} \ = \ 
\begin{array}{cccc}
a_{1,1}&a_{1,2}&\dots&a_{1,N}\\
\vdots&\vdots&&\vdots\\
a_{r,1}&a_{r,2}&\dots&a_{r,N}\\
a'_{1,1}&a'_{1,2}&\dots&a'_{1,N}\\
\vdots&\vdots&&\vdots\\
a'_{s,1}&a'_{s,2}&\dots&a'_{s,N}\\
\end{array}\;,
\ee 
assuming that $r+s\leq d^N$. Let us remark that $\oplus$ is only a \emph{partial} composition function, in contrast to a \emph{total} one, in the sense it is well-defined only for two OAs such that their number of runs is not greater than $d^N$. The operation $\oplus$ admits an identity element, which is the empty array (having $r=0$ rows); moreover, it is associative and abelian. 

The set $\mathscr{O\!A}(N,d,k)$ contains a subset of OAs, referred to as \emph{generating OAs}, which cannot be written as the composition of any two OAs (unless trivially, when one of them is the empty array), while every other OA, which is not in such subset, can be written as a composition of two or more generating OAs. In this way, the generating OAs allow for a compact description of the entire set $\mathscr{O\!A}(N,d,k)$. 

The following method to compute the generating OAs was developed 
by Pistone et al.~\cite{pistone1996generalised,pistone2000algebraic,carlini2007hilbert} 
from the point of view of algebraic statistics. Our starting point, instead, is the following connection between orthogonal arrays and quantum theory. 
Any OA($r,N,d,k$) encodes a pure quantum state $\ket{\Psi_{N,d,k}}\in{\mathcal H}^{\otimes N}_d$ of $N$ qudits, with $\mathcal H_d\simeq \mathbb C^d$ denoting the local Hilbert space of each subsystem, via the following mapping \cite{goyeneche2014genuinely}:
\begin{equation}\label{OAPsi}
\tau:\quad OA(r,N,d,k)=
\begin{array}{cccc}
a_{1,1}&a_{1,2}&\dots&a_{1,N}\\
a_{2,1}&a_{2,2}&\dots&a_{2,N}\\
\vdots&\vdots&&\vdots\\
a_{r,1}&a_{r,2}&\dots&a_{r,N}\\
\end{array}\hspace{0.5cm}\Longrightarrow\hspace{0.5cm}
\ket{\Psi_{N,d,k}}=\sum_{i=1}^{r}|a_{i,1}a_{i,2},\dots,a_{i,N} \rangle\in
%%{(\mathbb{N}_0^{d})}^{\otimes N},
{(\mathcal H_d)}^{\otimes N},
\end{equation}
where 
%$\mathbb{N}_0=\mathbb{N}\cup\{0\}$,
 $|a_{i,1}a_{i,2},\dots,a_{i,N}\rangle$ denotes the product state
$|a_{i,1}\rangle\otimes|a_{i,2}\rangle\otimes\dots\otimes|a_{i,N}\rangle$,
 and $\otimes$ stands for the Kronecker product. 
For example,
\begin{equation}
\tau:\quad OA(2,2,2,1)=
\begin{array}{cccc}
0&0\\
1&1
\end{array}\hspace{0.5cm}\Longrightarrow\hspace{0.5cm}
\ket{\Psi_{2,2,1}}=\ket{00}+\ket{11}\;,
\end{equation}
leads to the Bell state $\ket{\Phi^+}$ of two qubits. 
For convenience, we consider unnormalized 
pure quantum states $\ket{\Psi_{N,d,k}}$ along our work.

For a given set of orthogonal arrays $\mathscr{O\!A}(N,d,k)$, its image under the map $\tau$ will be denoted by $\mathcal{O\!A}(N,d,k)$; it is the set of corresponding \emph{array-based} quantum states belonging to $\mathcal H_d^{\otimes N}$. It includes a variety of states useful for different tasks of quantum information processing. In particular, the subset of states generated by \emph{irredundant} orthogonal arrays,  written IrOA,  plays a special role in quantum mechanics \cite{goyeneche2014genuinely,goyeneche2016multipartite}. An OA($r,N,d,k$) is called irredundant if every subarray made up of $N-k$ columns contains no repeated rows. Any IrOA($r,N,d,k$) defines a multipartite pure state of $N$ parties, with $d$ levels each, such that every reduction to $k$ parties is maximally mixed. Therefore entanglement with respect to any splitting of the $N$ subsystems into subsets of size $k$ and $N-k$ is maximal; such states are are called $k$-uniform. 

In particular, an irredundant array with index unity, written IrOA($d^k,N,d,k$), defines a $k$-uniform state with minimal support. Such arrays are completely equivalent to maximum distance separable (MDS) codes \cite{S04}. Speaking intuitively, their rows can be used to efficiently protect information against the presence of errors. MDS codes have the property that the Hamming distance between any two codewords, i.e.~any two rows of an irredundant array with index unity, is a constant taking the maximal allowed value (according to the Singleton bound). A comprehensive introduction to codes and their relation to orthogonal arrays can be found in Chapters 4 and 5 of Ref.~\cite{hedayat1999orthogonal}.

From Eq.~\eqref{OAPsi}, any OA($r,N,d,k$) is uniquely associated to an $N$-partite pure quantum state $\ket{\Psi_{N,d,k}}$ which can be rewritten as follows,
\begin{equation}\label{Psi}
\ket{\Psi_{N,d,k}}=\sum_{i_1,\dots,i_N=0}^{d-1}c_{i_1\cdots i_N}\,|i_1\cdots i_N\rangle\;,
\end{equation}
where the coefficients $c_{i_1\cdots i_N}$ are nonnegative integers counting how many times the run $i_1\cdots i_N$ occurs along the orthogonal array OA($r,N,d,k$). Eq.~(\ref{Psi}) allows one to express the requirement of orthogonality of the array as a linear system of equations for the coefficients $c_{i_1\cdots i_N}$ of $\ket{\Psi_{N,d,k}}$. 

For instance, consider the set $\mathscr{O\!A}(2,2,1)$. Taking into account Eq.~(\ref{Psi}), one has that the coefficients of the state 
\be
\ket{\Psi_{2,2,1}}=c_{00}\ket{00}+c_{01}\ket{01}+c_{10}\ket{10}+c_{11}\ket{11}\;.
\ee
are constrained to satisfy 
\begin{eqnarray}\label{SOE21}
c_{00}+c_{01}=c_{10}+c_{11}\;,\nonumber\\
c_{00}+c_{10}=c_{01}+c_{11}\;.
\end{eqnarray}
Eq.~(\ref{SOE21}) implies that the same number of zeros and ones occurs in each column of the array. In addition, one has to impose the constraint
\be\label{cstrff}
c_{00}+c_{01}+c_{10}+c_{11}\leq 4\;,
\ee
so that the total number of runs is not greater than the number of runs for the full factorial design. As another example, consider the case of $\mathscr{O\!A}(3,2,2)$. Orthogonality of the array requires that every possible combination of two symbols taken from the set $\{0,1\}$ occurs the same number of times in every pair of columns. Imposing such restrictions to the state
\begin{eqnarray}
\ket{\Psi_{3,2,2}}&=&c_{000}\ket{000}+c_{001}\ket{001}+c_{010}\ket{010}+c_{011}\ket{011}+\nonumber\\
&&c_{100}\ket{100}+c_{101}\ket{101}+c_{110}\ket{110}+c_{111}\ket{111},
\end{eqnarray}
leads to the linear constraints 
 \begin{eqnarray}\label{EQ32}
&c_{000}+c_{001}=c_{010}+c_{011}=c_{100}+c_{101}=c_{110}+c_{111}\label{EQ32_1}\;,\nonumber\\
&c_{000}+c_{010}=c_{001}+c_{011}=c_{100}+c_{110}=c_{101}+c_{111}\label{EQ32_2}\;,\nonumber\\
&c_{000}+c_{100}=c_{001}+c_{101}=c_{010}+c_{110}=c_{011}+c_{111}\label{EQ32_3}\;.
\end{eqnarray}
There is also a further constraint similar to that of Eq.~\eqref{cstrff}.

More generally, the following result can be established:
\begin{prop}\label{prop1}
Any orthogonal array in $\mathscr{O\!A}(N,d,k)$ is uniquely associated to a pure quantum state 
\be
\ket{\Psi_{N,d,k}}=\sum_{i_1,\dots,i_N=0}^{d-1}c_{i_1\cdots i_N}\,|i_1\cdots i_N\rangle ,
\ee
with coefficients $c_{i_1\cdots i_N}$ satisfying the following set of linear constraints
\begin{equation}\label{LC}
\sum_{i_1,\dots,i_{N-k}=0}^{d-1} c_{\sigma(i_1\cdots i_{N-k},\,j_1\cdots j_k)}= \sum_{i_1,\dots,i_{N-k}=0}^{d-1}c_{\sigma(i_1\cdots i_{N-k},\,j'_1\cdots j'_k)}\;,
\end{equation}
together with the further constraint
\be\label{cstr}
\sum_{i_1,\dots,i_N=0}^{d-1}c_{i_1\cdots i_N} \leq d^N\;.
\ee
Here, $\sigma \in S_N$ is any permutation on the space of multi-indices of length $N$, and $j_1\cdots j_k$, $j'_1\cdots j'_k$ are any two multi-indices of length $k$.  
\end{prop}
\begin{proof}
Consider an arbitrary combination of $k$ symbols, denoted by $j_1\dots j_k$, where each symbol is taken from the alphabet $\{0,\,1,\dots,\,d-1\}$. Then, the total number of times that such a combination appears in the last $k$ columns of a given array, denoted by $\#(j_1\dots j_k)$, is obtained by summing over all possible combinations of symbols in the first $N-k$ columns of the array, i.e. 
\be\label{js}
\#(j_1\dots j_k)=\sum_{i_1,\dots,i_{N-k}=0}^{d-1}c_{i_1\cdots i_{N-k},\,j_1\cdots j_k}\;.
\ee
By definition of orthogonal array, any combination of $k$ symbols occurs the same number of times in any $k$ columns. In particular, this holds for the last $k$ columns. Using Eq.~(\ref{js}) we have 
\be
\#(j_1\dots j_k)=\sum_{i_1,\dots,i_{N-k}=0}^{d-1}c_{i_1\cdots i_{N-k},\,j_1\cdots j_k}=\sum_{i_1,\dots,i_{N-k}=0}^{d-1}c_{i_1\cdots i_{N-k},\,j'_1\cdots j'_k}=\#(j'_1\dots j'_k)\;.
\ee
Applying a similar condition for \emph{any} possible choice of $k$ columns leads us to Eq.~\eqref{LC}. 
\end{proof}
The linear system of Eq.~\eqref{LC} imposes a total of $(d^k-1)\times N!/[k!(N-k)!]$  constraints on the coefficients of the state $\ket{\Psi_{N,d,k}}$. If the coefficients are allowed to take nonnegative real values, each linearly independent constraint defines an hyperplane, restricted to the positive orthant of ${(\mathbb{R}^{d})}^{\otimes N}$. The intersection among all hyperplanes gives rise to a rational cone, denoted by $C(N,d,k)$, which is made bounded by the further constraint of Eq.~\eqref{cstr}. The resulting convex polytope is denoted by $P(N,d,k)$. Points belonging to $ P(N,d,k)$ and having nonnegative integer coordinates are one-to-one related to orthogonal arrays in $\mathscr{O\!A}(N,d,k)$. We therefore arrive at the following result:
{\begin{prop}\label{prop2}
Any orthogonal array in $\mathscr{O\!A}(N,d,k)$ is uniquely associated to a point 
of the convex polytope $P(N,d,k)$ having nonnegative integer coordinates, where $ P(N,d,k)$ is defined by the set of linear constraints \eqref{LC} and \eqref{cstr}.
\end{prop}
The set $\mathscr{O\!A}(N,d,k)$ can thus be visualized as a $d^N$-dimensional integer lattice. It can be fully characterized with the help of the Gordan's lemma \cite{bruns2010polytopes}, a well-known result from invariant theory. The lemma states that any integer point of a rational cone can be represented as a combination with nonnegative integer weights of a finite number of points. The requirement that the set of such points is as small as possible makes it unique. It is usually referred to as the Hilbert basis of the cone \cite{hilbert1993theory}. 

We introduce the following terminology. The arrays in the set $\mathscr{O\!A}(N,d,k)$, which 
are represented by points of $P(N,d,k)$ belonging to the Hilbert basis of $C(N,d,k)$, are referred to as the \emph{generating OAs}, while the quantum states related to them are called \emph{generating states}. For any given values of the parameters $N$, $d$ and $k$, we denote the family of all generating arrays by $\mathscr G(N,d,k)$ and the corresponding set of quantum states by $\mathcal G(N,d,k)$. The discussion above can be then summarized through the following proposition,
\begin{prop}\label{prop3}
$\mathscr{O\!A}(N,d,k)$ consists of all OAs of the form 
\be
\text{OA($r,N,d,k$)} = \bigoplus_{i=1}^m \alpha_i\, \text{OA$^{(i)}_{N,d,k}$}\;,
\ee
where $\alpha_i \in \mathbb N_0$, $\mathbb{N}_0\eqd \mathbb{N}\cup\{0\}$, $m$ is the cardinality of the Hilbert basis of the rational cone $C(N,d,k)$ and \text{OA$^{(i)}_{N,d,k}$} is the generating OA associated to the $i^{\text{th}}$ element of the Hilbert basis. The numbers $\alpha_i$ are 
constrained to satisfy $\sum_{i=1}^m \alpha_i\, r^{(i)} \leq d^N$, where $r^{(i)}$ is the number of runs of \text{OA$^{(i)}_{N,d,k}$}. 

At the same time, the set $\mathcal{O\!A}(N,d,k)$ consists of all states of the form
\begin{equation}
\ket{\Psi_{N,d,k}} = \sum_{i=1}^m \alpha_i\,\ket{g_{N,d,k}^{(i)}}\;,
\end{equation}
where $\ket{g_{N,d,k}^{(i)}}$ is the generating state based on \text{OA$^{(i)}_{N,d,k}$} and the $\alpha_i$ are constrained as above.
\end{prop}

In Fig.~\ref{Fig1}, we illustrate Proposition \ref{prop3} by showing a three-dimensional projection of the convex polytope $P(3,2,1)$, with the lattice points on its faces representing orthogonal arrays OA$(r,3,2,1)$. Let us also remark that the problem of computing the Hilbert basis of a rational cone is well-studied, with relevant applications in different branches of mathematics, such as combinatorics \cite{weismantel1996hilbert}, integer programming \cite{graver1975foundations} and computational algebra \cite{cox1992ideals}. Specialized software is available, like the open source tool \emph{Normaliz} \cite{bruns2010normaliz}. Though in general computationally hard \cite{durand1999complexity}, for small values of the parameters $N$, $d$ and $k$, the problem can be readily solved, providing a compact representation of $\mathscr{O\!A}(N,d,k)$.
 
%To continue  previous examples, we 
Let us return back to the illustrative set of arrays denoted by $\mathscr{O\!A}(2,2,1)$. The Hilbert basis of $C(2,2,1)$, defined by the linear constraints of Eq.~\eqref{SOE21}, consists of the following two generating OAs
\be\label{OA221}
\text{OA$^{(1)}_{2,2,1}$}=\ba{cc}
0&0\\
1&1
\ea\;,\qquad\qquad
\text{OA$^{(2)}_{2,2,1}$}=\ba{cc}
0&1\\
1&0
\ea\;,
\ee 
with corresponding states 
\begin{equation}\label{gen21}
|g_{2,2,1}^{(1)}\rangle=|00\rangle+|11\rangle\;,\qquad|g_{2,2,1}^{(2)}\rangle=|01\rangle+|10\rangle,
\end{equation}
which, up to normalization, are equivalent to the 
Bell states usually denoted by $\ket{\Phi^+}$ and $\ket{\Psi^+}$. 
Proposition \ref{prop3} implies that any 
array-based state in $\mathcal{O\!A}(2,2,1)$ is a superposition with nonnegative integer coefficients of the generating states $\ket{g_{2,2,1}^{(1)}}$ and $\ket{g_{2,2,1}^{(1)}}$. 

Finally, let us consider the set $\mathscr{O\!A}(3,2,2)$. Computation of the Hilbert basis of the cone $C(3,2,2)$, defined by the linear constraints \eqref{EQ32}, leads to the arrays
\be\label{oag32}
\text{OA$^{(1)}_{3,2,2}$}=\ba{ccc}
0&0&1\\
0&1&0\\
1&0&0\\
1&1&1
\ea\;,\qquad\qquad
\text{OA$^{(2)}_{3,2,2}$}=\ba{ccc}
0&0&0\\
0&1&1\\
1&0&1\\
1&1&0
\ea\;,
\ee
with  corresponding states
\be
\label{g32}
\ket{g_{3,2,2}^{(1)}} = \ket{001} +\ket{010} + \ket{100}+ \ket{111}
\;,\qquad
\ket{g_{3,2,2}^{(2)}} = \ket{000}+ \ket{011}+ \ket{101}+ \ket{110}\;.
\ee

The full list of generating OAs for systems of up to four qubits is reported in Appendix \ref{app1}. In addition, the Hilbert bases associated to $\mathscr{O\!A}(5,2,k)$ for $k=1,\dots,4$, as well as for heterogeneous tripartite systems of two qubits and one qutrit or one qubit and two qutrits, are available online \cite{HBonline}. All bases were calculated via the open source software \emph{Normaliz} \cite{bruns2010normaliz} on a 2.9 GHz Intel Core i5 Processor with 8 GBs of memory. The elapsed time for the largest basis (corresponding to the case of five qubits and $k=2$) was 56 seconds. 

While there is in general no apparent regularity in the Hilbert bases corresponding to different values of $N$, $d$ and $k$, for maximum strength, i.e.~$k=N-1$, the following conjecture can be formulated:
\begin{conje}\label{conjecture}
Consider the set of orthogonal arrays $\mathscr{O\!A}(N,2,N-1)$. For every $N$, the Hilbert basis of the rational cone $C(N,2,N-1)$ consists of the following two states:
\be
\ket{g_{N,2,N-1}^{(1)}}=\underbrace{\mathbb I \otimes \dots \otimes \mathbb I}_\text{$N-1$ times}\otimes \,\sigma_x\,\ket{g_{N,2,N-1}^{(2)}}\;,\qquad \ket{g_{N,2,N-1}^{(2)}} = \underbrace{H\otimes \dots\otimes H}_\text{$N$ times}\,\ket{GHZ_N}\;.
\ee
\end{conje}
Here, $\mathbb{I}_2$ denotes the identity matrix of size 2,  $\sigma_x=\bigl( \begin{smallmatrix}0 & 1\\ 1 & 0\end{smallmatrix}\bigr)$ is the  Pauli  shift matrix, and $H$ is the unnormalized Hadamard gate, $H=\bigl( \begin{smallmatrix}1 & 1\\ 1 & -1\end{smallmatrix}\bigr)$. This conjecture is supported by analytic computation of the Hilbert bases associated to the respective cones up to $N=10$.

\begin{figure}
\includegraphics[width=0.65\textwidth]{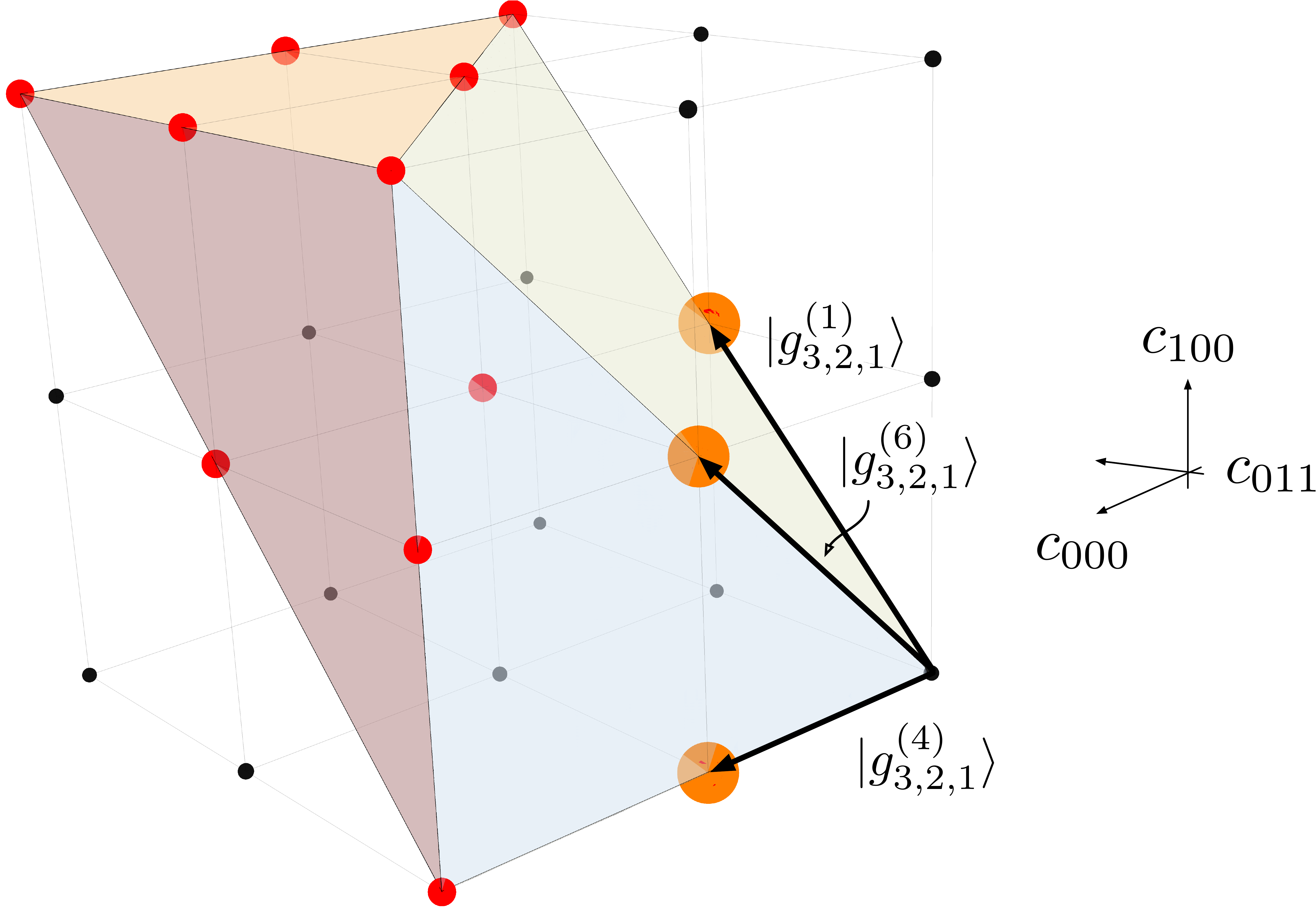}
\caption{\emph{(color online)} Three-dimensional projection of the convex polytope $ P(3,2,1)$. The projection is obtained by imposing the constraint $c_{001}=c_{010}=0$ and restricting the polytope to the box $[0,2]^{\times 3}$. It contains three of the six generating states of the corresponding Hilbert basis, shown in orange \emph{(larger points)}. Red \emph{(medium)} points define quantum states associated to OAs that are convex superpositions of the generating states. Black \emph{(smaller)} points in the lattice are outside the cone, and thus are not associated to any orthogonal array.}
\label{Fig1}
\end{figure}

We conclude this section with an application to the theory of orthogonal arrays. The previous results can be used to prove nonexistence of some classes of OAs in a direct way. For instance, by exploring the Hilbert basis of the cone $C(4,2,2)$, as reported in Appendix \ref{app1}, one realizes that all generating OAs with $N=4$ and $k=2$ over a binary alphabet have at least 8 runs. Therefore, OA($4,4,2,2$) does not exist, which is equivalent to saying that two orthogonal Latin squares of size 2 do not exist. One can reach the same conclusion without the necessity to derive the full Hilbert basis. That is, from the system of linear constraints defining the cone $C(4,2,2)$ one can, without loss of generality, assume that $c_{0000}=1$, so that by inversion one finds $c_{1111} =(r-6)/2-c_{0001}-c_{0010}-c_{0100}-c_{1000}$. But then, for $r=4$, the coefficient $c_{1111}$ would be negative, which is impossible. In a similar way, one could  prove that an OA($36,4,6,2$), or two orthogonal Latin squares of size six, do not exist. This statement is equivalent to the fact that the famous problem of Euler \cite{euler1782recherches} concerning 36 officers has no solution.

\section{Entanglement classification for array-based states}
\label{S4}
In this section, we introduce the problem of entanglement classification in the set
 $\mathcal{O\!A}(N,d,k)$. The first issue is to define the class of \emph{free operations}. 
In the Hilbert space $\mathcal H^{\otimes N}_d$, one usually considers the so called SLOCC class, which corresponds to the group of local invertible transformations, $\text{GL($d$)}^{\otimes N}$. However, in general, a local invertible transformation is not OA-preserving, i.e.~it does not map an array-based state into another array-based state.
\begin{prop}\label{prop4}
The most general OA-preserving transformation is the local composition of transformations represented by invertible integer stochastic matrices.
\begin{proof}
Consider an invertible transformation acting on one party only. This implies that, if $\ket{\psi_{N,d,k}}$ is an array-based quantum state, it is transformed to $T_{A_1}\ket{\psi_{N,d,k}} \vcentcolon= (A_1\otimes \mathbb I_d \otimes \dots \otimes \mathbb I_d)\ket{\psi_{N,d,k}}$, with $A_1\in \text{GL($d_i$)}$
and  without loss of generality the transformation is applied to the first party.
For the transformation $T_{A_1}$ to be OA-preserving, any two entries of $A_1$ must be in rational proportions. Since states are unnormalized, we may assume that all entries are integer. If some of the entries are negative, then $T_{A_1} \ket{\psi_{N,d,k}}$ is in general not array-based, thus we may further assume that all entries are nonnegative. Let us denote by $\alpha_{ij}\in \mathbb N_0$ the entries of $A_1$, with $i,j \in \{0,\dots,d-1\}$. Then, by construction,
\be\label{ta1trans}
T_{A_1}\ket{j}\otimes \ket{\phi} = \sum_{i=0}^{d-1} \alpha_{ij} \ket{i}\otimes \ket{\phi}\;,
\ee 
for arbitrary $\ket{\phi}\in \mathcal H_d^{\otimes N-1}$. Assuming that $T_{A_1}$ is OA-preserving implies that
\be
\sum_{j=0}^{d-1} \alpha_{i_1 j}= \sum_{j=0}^{d-1} \alpha_{i_2j}\;,
\ee
for any choice of row indices $i_1$ and $i_2$, 
so each symbol $s\in\mathscr A_d$ appears equally often along the first column of the array corresponding to $T_{A_1}\ket{\psi_{N,d,k}}$. Moreover, it is also necessary that each symbol appears equally often in any other column, which requires that
\be
\sum_{i=0}^{d-1} \alpha_{i j_1}= \sum_{i=0}^{d-1} \alpha_{ij_2}\;,
\ee
for any choice of column indices $j_1$ and $j_2$. In conclusion, $A_1$ must have nonnegative integer entries that sum to the same constant $c$ along either a row or a column, i.e.~$A_1$ is an integer stochastic matrix or a generalized magic square with magic constant $c$ \cite{stanley2007combinatorics}. In turn, if $A_1$ is an integer stochastic matrix, then $T_{A_1} \ket{\psi_{N,d,k}}$ is array-based, since the corresponding array is an OA with strength $k$ at least $1$. 
The most general OA-preserving transformation is obtained by composition, i.e.~it is of the form $\sbigotimes_j\!T_{A_j}$.
\end{proof}
\end{prop}

As an example, let us consider the case of qubits, that is $d=2$. If $A_j$ denotes a general invertible matrix,
\be
A_j = \begin{pmatrix}
\alpha_{00} & \alpha_{01} \\
\alpha_{10} & \alpha_{11}
\end{pmatrix}\;,\qquad \alpha_{00}\alpha_{11} - \alpha_{01}\alpha_{10} \neq 0\;,
\ee
with $\alpha_{ij} \in \mathbb N_0$, for $A_j$ to be integer stochastic, one must have:
\be
\begin{split}
\alpha_{00} + \alpha_{10} &= \alpha_{01} + \alpha_{11}\;,\\
\alpha_{00} + \alpha_{01} &= \alpha_{10} + \alpha_{11}\;,\\
\end{split}
\ee
which implies $\alpha_{00}=\alpha_{11}\eqd\alpha$ and $\alpha_{01}=\alpha_{10}\eqd \beta$, with $\alpha\neq \beta$. Thus, the most general OA-preserving transformation for qubits is the composition of transformations $T_{A_j}$, with 
\be
A_j = \begin{pmatrix}
\alpha & \beta \\
\beta & \alpha
\end{pmatrix}\;,\qquad \alpha \neq \beta\;.
\ee
For instance, if $\ket{\psi_{N,2,k}}$ is the state based on OA($r,N,2,k$), then 
\be
\ket{\psi_{N,2,k}} \to T_{A_1} \ket{\psi_{N,2,k}} = \alpha \ket{\psi_{N,2,k}} + \beta\,(\sigma_x \otimes \mathbb I_2 \dots \otimes \mathbb I_2) \ket{\psi_{N,2,k}}\;,
\ee
or, at the level of arrays,
\be\label{trOA}
\text{OA($r,N,2,k$)}  \to \alpha\,\text{OA($r,N,2,k$)} \oplus \beta\, f_1[\text{OA($r,N,2,k$)}]\;,
\ee
where $f_j[OA]$ is obtained from $OA$ by permuting symbols along the $j^{\text{th}}$ column. It is clear that the transformation \eqref{trOA} does not change the strength of the orthogonal array it is applied to. In fact, the following more general proposition holds:
\begin{prop}\label{kprop}
Any two OAs, which are pre-images with respect to the map $\tau$ 
of quantum states connected by an OA-preserving transformation, have the same strength.
\begin{proof}
It is enough to check that, if $\ket{\psi_{N,d,k}}$ is based on OA($r,N,d,k$), then $T_{A_1}\ket{\psi_{N,d,k}}$ is based on an OA having the same strength $k$. 
Here $A_1$ represents an integer stochastic matrix, 
with entries $\alpha_{ij}\in\mathbb N_0$ and magic constant $c=\sum_{i=0}^{d-1}\alpha_{ij}=\sum_{j=0}^{d-1} \alpha_{ij}$. For convenience, we introduce the following notation. Given a multi-index $I=j_1j_2\dots j_k$ with $|I|=k$, let  $\pi_I$ be
 the linear operator that acts on basis kets as follows:
 $\pi_I \ket{i_1 i_2\dots i_N}=\ket{i_{j_1} i_{j_2}\dots i_{j_k}}$.
 For instance, $\pi_{23}\ket{0123}=\ket{12}$. 
Furthermore, if $s\in\mathscr A_d$, then $\zeta_s$ is the linear operator that acts on basis kets by $\zeta_s\ket{i_1 i_2\dots i_N}=\delta_{s,i_1}\ket{i_2\dots i_N}$. For instance, $\zeta_0 \ket{0123} = \ket{123}$, but $\zeta_1 \ket{0123}=0$. 

The fact that $\ket{\psi_{N,d,k}}$ is based on an OA of strength $k$ is equivalent to the fact that, for any $I$ with $|I|=k$, 
\be\label{p1}
\pi_I\ket{\psi_{N,d,k}}=\lambda\, \tau(\mathscr F_{k,d})\;, 
\ee
i.e.~the projection of OA($r,N,d,k$) to any $k$ columns is a multiple of the full factorial design $\mathscr F_{k,d}$, with $\lambda$ being the \emph{index} of OA($r,N,d,k$). Moreover, it also holds that, for any $I$ with $|I|=k-1$, and any $s\in\mathscr A_d$,
\be\label{p2}
\pi_I  \circ \zeta_s \ket{\psi_{N,d,k}} = \lambda\, \tau(\mathscr F_{k-1,d})\;.
\ee 

We have to check that $T_{A_1}\ket{\psi_{N,d,k}}$ is based on an OA of strength $k$. Write $\ket{\psi_{N,d,k}}$ in the following form, 
\be
\begin{split}
\ket{\psi_{N,d,k}} =& \sum_{l=1}^{r/d} \ket{0}\otimes \ket{\phi_{0l}} + \sum_{l=1}^{r/d} \ket{1}\otimes \ket{\phi_{1l}} + \dots +  \sum_{l=1}^{r/d} \ket{d-1}\otimes \ket{\phi_{d-1,l}} 
=  \sum_{l=1}^{r/d} \sum_{j=0}^{d-1} \ket{j}\otimes \ket{\phi_{jl}}\;, 
\end{split}
\ee
where the $\ket{\phi_{jl}}$ are basis kets. Then, from Eq.~\eqref{ta1trans}, it follows that
\be\label{ta1state}
T_{A_1}\ket{\psi_{N,d,k}} = \sum_{l=1}^{r/d} \sum_{j=0}^{d-1} \sum_{i=0}^{d-1} \alpha_{ij} \ket{i} \otimes \ket{\phi_{jl}}\;.
\ee
Suppose $k$ is at least $2$ (otherwise the statement is trivial). We have to check that for every $I$ with $|I|=2$, the state
$\pi_I\ket{\psi_{N,d,k}}$ is a multiple of $\tau(\mathscr F_{2,d})$. For instance, for $I=23$
one has 
\be
\begin{split}
\pi_{23} \circ T_{A_1}\, \ket{\psi_{N,d,k}} = & \sum_{l=1}^{r/d} \sum_{j=0}^{d-1}\sum_{i=0}^{d-1} \alpha_{ij}\, \pi_{23}(\ket{i}\otimes \ket{\phi_{jl}})
= \sum_{l=1}^{r/d} \sum_{j=0}^{d-1} \left(\sum_{i=0}^{d-1} \alpha_{ij}\right) \pi_{12}\ket{\phi_{jl}}\\
= &\, c \sum_{l=1}^{r/d} \sum_{j=0}^{d-1} \pi_{12}\ket{\phi_{jl}} 
= \, c\, \pi_{23} \ket{\psi_{N,d,k}} \\
= & \, \lambda c\, \tau(\mathscr F_{2,d})\;.
\end{split}
\ee
It is clear that the same holds for any choice of $I$ not involving the first column. Consider instead $I=12$; then,
\begin{align*}
\pi_{12} \circ T_{A_1}\, \ket{\psi_{N,d,k}} = & \sum_{l=1}^{r/d} \sum_{j=0}^{d-1}\sum_{i=0}^{d-1} \alpha_{ij}\, \pi_{12}(\ket{i}\otimes \ket{\phi_{jl}})
= \sum_{l=1}^{r/d} \sum_{j=0}^{d-1} \sum_{i=0}^{d-1} \alpha_{ij}\ket{i}\otimes \pi_1 \ket{\phi_{jl}}\\
= & \sum_{i=0}^{d-1} \sum_{j=0}^{d-1} \alpha_{ij}\ket{i}\otimes \pi_1\left(\sum_{l=1}^{r/d} \ket{\phi_{jl}} \right)
=  \sum_{i=0}^{d-1} \sum_{j=0}^{d-1} \alpha_{ij}\ket{i}\otimes \pi_1 \circ \zeta_j \ket{\psi_{N,d,k}} \tag{\stepcounter{equation}\theequation} \\
= &\, \lambda\, \sum_{i=0}^{d-1} \left(\sum_{j=0}^{d-1}\alpha_{ij}\right) \,\ket{i}\otimes \tau(\mathscr F_{1,d})
= \, \lambda c \sum_{i=0}^{d-1} \ket{i} \otimes \tau(\mathscr F_{1,d})\\
= &\, \lambda c\, \tau(\mathscr F_{2,d})\;.
\end{align*}
The same holds for any choice of $I$ involving the first column and any other column. 

If $k$ were equal to 2, this would conclude the proof. If $k>2$, it is enough to notice that the previous computations rely on the two properties \eqref{p1} and \eqref{p2}, which hold for any $k$.
\end{proof}
\end{prop}

Thus, OA-preserving transformations do not change the strength $k$. Moreover, they can only increase the generalized resolution of the corresponding array:
\begin{prop}\label{grprop}
If a quantum state $|\psi\rangle$  is obtained from another array-based state $|\phi\rangle$
 by an OA-preserving transformation, then the OA corresponding to  $|\psi\rangle$ 
has generalized resolution not smaller than the array corresponding to  $|\phi\rangle$.
\begin{proof}
Define $\Omega$ to be the linear functional which acts on basis kets as follows: $\Omega \ket{i_1i_2 \dots i_n}= \omega_{i_1} \omega_{i_2} \dots \omega_{i_n}$, where as before $\omega_s= \text{exp}(2\pi i s/d)$ for $s\in\mathscr A_d$. Given the array-based state $\ket{\Psi_{N,d,k}}$, the $J$-characteristic $J_n(I)$ of the OA $\tau^{-1}\ket{\Psi_{N,d,k}}$ with multi-index $I$ can be written as
\be
J_n(I) = \left| \Omega \circ \pi_I \ket{\psi_{N,d,k}}\right|\;.
\ee  
Suppose the transformation $T_{A_1}$ is applied to $\ket{\psi_{N,d,k}}$; if our statement holds in this particular case, it also holds for the most general OA-preserving transformation. The $J$-characteristic of the OA $\tau^{-1} \circ T_{A_1} \ket{\psi_{N,d,k}}$ with multi-index $I$ is denoted by $\tilde J_n(I)$. There are two cases: either $I$ includes the first index or not. If not, from Eq.~\eqref{ta1state}, it follows that:
\be\label{case1jn}
\begin{split}
\tilde J_n(I) =& \left|\Omega \circ  \pi_I \circ T_{A_1} \ket{\psi_{N,d,k}}\right| = \left| \sum_{l=1}^{r/d} \sum_{j=0}^{d-1}\left(\sum_{i=0}^{d-1}\alpha_{ij}\right)\Omega \circ \pi_I \ket{\phi_{jl}}\right| \\
= &\,c\,\left| \Omega \circ \pi_I \ket{\psi_{N,d,k}}\right| = c \, J_n(I)\;.
\end{split}
\ee 
We conclude that any $J$-characteristic $\tilde J_n(I)$, with $I$ not involving the first column, is rescaled by $c$, the magic constant of the matrix $A_1$. 

Now suppose $I$ includes the first index, so that it may be written as $I = 1\cdot I'$, where $|I'|=n-1$. Then,
\be
\begin{split}
\tilde J_n(I) =& \left| \sum_{l=1}^{r/d} \sum_{j=0}^{d-1}\sum_{i=0}^{d-1}\alpha_{ij}\, \omega_i\, \Omega \circ \pi_{I'} \ket{\phi_{jl}}\right| 
= \left|\sum_{j=0}^{d-1} \left(\sum_{i=0}^{d-1}\omega_i\,\alpha_{ij}\right) \left(\sum_{l=1}^{r/d} \Omega \circ \pi_{I'} \ket{\phi_{jl}}\right)\right| \\
=&\, c \left|\sum_{i,j=0}^{d-1} \omega_i \beta_{ij} z_j \right|\;,
\end{split}
\ee 
where we defined
\be\label{jtilden}
\beta_{ij} \eqd \frac{\alpha_{ij}}{c}\;, \qquad\qquad  z_j \eqd \sum_{l=1}^{r/d} \Omega \circ \pi_{I'} \ket{\phi_{jl}}\;.
\ee
Notice that $\beta_{ij}$ is a bistochastic matrix of order $d$. In matrix notation, we may rewrite the summation on the right hand side of Eq.~\eqref{jtilden} as $\pmb{\omega}^T \beta \mathbf{z}$, where $\pmb{\omega}$ and $\mathbf{z}$ are complex column vectors and $\bullet^{T}$ denotes the
matrix transposition. By Birkhoff theorem, the set of bistochastic matrices 
is given by the convex hull of all permutations of given size.
Hence we can represent an arbitrary bistochastic matrix as a combination
of permutation matrices $\sigma_q$. In particular one can write, 
$\beta = \sum_{q} \theta_q \sigma_q$, where the positive weights $\theta_q$ sum to unity, 
%$\theta_q >0$, 
$\sum_q \theta_q = 1$. Then,
\be
\tilde J_n(I) = c \left| \sum_q \theta_q\, \pmb{\omega}^T \sigma_q \mathbf{z} \right| \leq c \sum_q \theta_q\,|\pmb{\omega}^T \sigma_q \mathbf{z}|\;.
\ee
Notice that $|\pmb{\omega}^T \sigma_q \mathbf{z}|$ is the $J$-characteristic, for the same multi-index $I$, of the OA $\tau^{-1} \circ T_{\sigma_q^T} \ket{\psi_{N,d,k}}$ (which is isomorphic to the OA $\tau^{-1} \ket{\psi_{N,d,k}}$). This implies that 
\be
\tilde J_n(I) \; \leq \; c \mu\; ,\qquad \quad \text{with} \qquad \mu \eqd \underset{\sigma}{\text{max}}\, |\pmb{\omega}^T \sigma_q \mathbf{z}|\;.
\ee
We conclude that any $J$-characteristic $\tilde J_n(I)$, with $I$ involving the first column, is less or equal to $c$ times the maximum of the $J$-characteristics, for the same multi-index $I$, of OAs obtained from $\tau^{-1} \ket{\psi_{N,d,k}}$ by a permutation of symbols in the first column. 

In particular, our results imply that: if $J_n(I)=0$, with $I$ not involving the first index, then also $\tilde J_n(I)=0$; if $I$ involves the first index and $\mu=0$, then also $\tilde J_n(I)=0$. This in turn means that the quantity $t$ appearing in definition \eqref{GRdef} can never decrease under OA-preserving transformations. In particular, if it increases, then, recalling that $t<GR<t+1$, also the generalized resolution increases. If it does not change, then one has to consider the maximum of the set of all $J$-characteristics of order $t$, taken over all possible multi-index $I$ with $|I|=t$ and permutations of symbols within columns. 

There are again two cases. Suppose that the multi-index achieving the maximum does not involve the first index. Then, from Eq.~\eqref{case1jn}, $J_t^{(\text{max})}$ gets rescaled by 
$c$. However, since the number of rows is also rescaled by the same factor, 
the generalized resolution remains unchanged.
Suppose instead that the multi-index achieving the maximum does not involve the first index. Then, we only know that $J_t^{(\text{max})}$ can not increase and, as a result, the resolution  \eqref{GRdef} can not decrease.
\end{proof}
\end{prop}

In light of Prop.~\ref{prop4}, it is natural to consider all OA-preserving transformations, as well as their inverses, to be free. Moreover, relabellings of the local Hilbert spaces are also free operations, since states thus obtained are physically indistinguishable. In conclusion, the class of free operations consists of arbitrary compositions of OA-preserving transformations and their inverses, as well as permutations of the local Hilbert spaces. Two array-based states are equivalent if they can be converted one into the other under free operations, as defined above. By definition, all states equivalent under free operations form an entanglement class in our classification. 

Let us remark that isomorphic OAs are associated to equivalent states. Indeed, given two isomorphic OAs, the transformation connecting their corresponding states is a composition of relabellings of the local Hilbert spaces (exchanges of columns) and local permutations (permutations of symbols within columns), i.e.~it is a free operation. However, it is important to highlight that non-isomorphic OAs could produce equivalent states. For instance, the GHZ state for three qubit systems is based on the array
\be\label{OAGHZ1}
\text{OA($2,3,2,1$)}=\begin{matrix}
0 & 0 & 0 \\
1 & 1 & 1 
\end{matrix}\;.
\ee
This state is equivalent to the state based on the array
\be\label{OAGHZ2}
\text{OA($6,3,2,1$)}=\begin{matrix}
0 & 0 & 0 \\
1 & 0 & 0 \\
1 & 0 & 0 \\
0 & 1 & 1 \\
0 & 1 & 1 \\
1 & 1 & 1 \\
\end{matrix}\;,
\ee
via the OA-preserving transformation $T_{A_1}$ with $A_1=\begin{psmallmatrix}1 & 2 \\ 2 & 1\end{psmallmatrix}$. However, from the point of view of orthogonal arrays theory, arrangements (\ref{OAGHZ1}) and (\ref{OAGHZ2}) are inequivalent.

Having formalized the entanglement classification problem in $\mathcal{O\!A}(N,d,k)$, we collect here a few simple results. Equivalent states in our classification are also equivalent according to the standard SLOCC classification, 
but the opposite does not hold. By Prop.~\ref{kprop}
 the strength $k$ is the same within each class.
Due to Prop.~\ref{grprop}, each entanglement class contains a state such that the corresponding array has the least number of runs, and thus also generalized resolution at least as small as any other array in the same class. Such a state is essentially unique, in the sense that all other states are based on isomorphic arrays. This is because an integer stochastic transformation increases the number of runs by a factor $c$, equal to its magic constant, and an integer stochastic transformation with $c=1$ is a permutation matrix. 
Thus, any two OA-based states in the same entanglement class having 
the same number of runs must be isomorphic. 

For each class, we call such a state a \emph{representative state}. The corresponding OA minimizes the number of runs within each class. Every state based on a generating array is a representative state (but the opposite does not necessarily hold). This is because a generating OA is always indecomposable, by definition, whereas all states which are not representative states are based on OAs which are decomposable. It follows that the states based on non-isomorphic arrays among all generating OAs belong to different entanglement classes. The remaining representative states must be searched for by taking all allowed compositions of the generating OAs and then identifying states connected by free operations; the OA with the least number of runs within each resulting class is the representative state of that class. Because there is only a finite number of allowed compositions, there is also a finite number of classes, for any number of parties $N$ and any local dimension $d$. 

For multipartite systems, there are several possible choices of entanglement measures, even for pure states \cite{HHHH09,EWZ16}. We choose to quantify entanglement by the mean von Neumann entanglement entropy averaged over all possible bipartitions. In other words, entanglement is quantified by the amount of information contained in different subsets of parties. For instance, pure multipartite states that can be prepared by using only local operations and classical communications are fully separable, i.e.~they can be written in tensor product form. For such states, the reduction to any number of parties gives rise to a pure state, which means that the total information in the global state is the sum of the information available to each party. However, for entangled states, this statement does not hold. 
An extreme case is given by the $k$-uniform states, for which all reductions to a fixed number of $k$ parties are maximally mixed. For instance, for the $1$-uniform GHZ state,
 the global state has vanishing von Neumann entropy, 
while the entropy of each single-party reduction is maximal. Therefore, each party has no information about the global state. This remarkable property of quantum states plays a fundamental role in quantum technologies like secure quantum communication \cite{BF02}, quantum cryptography \cite{GRT02} and randomness certification \cite{AM16}.

The array-based entanglement classification just introduced manages to capture many genuinely different types of entanglement. This is because, for any choice of the parameters $N$, $d$ and $k$, it always includes classes having the maximum and minimum amount of entanglement, as well as classes with different intermediate degrees of entanglement. In fact, 
\begin{itemize}
\item[\emph{(i)}] There is always an entanglement class corresponding to the fully separable case, which is based on the full-factorial $\mathscr F_{N,d}$. That is, the OA($d^N,N,d,N$) gives rise to the fully separable $N$-qudits state $|+\rangle^{\otimes N}$, with $\ket{+} = \ket{0} + \ket{1} + \dots +\ket{d-1}$. This state can be prepared by using only local operations and classical communication between the $N$ parties, and thus contains no quantum correlations.
\item[\emph{(ii)}] Irredundant OAs of strength $k$ give rise to $k$-uniform states \cite{goyeneche2014genuinely}. Since every irredundant OA has a number of runs $r<d^N$, all irredundant OAs of strength $k$ are in $\mathscr{O\!A}(N,d,k)$. In particular, states based on IrOA with the maximal strength $k\leq \lfloor N/2\rfloor$ are characterized by the highest possible degree of entanglement.
\end{itemize}

Therefore, the array-based classification covers the maximally entangled (corresponding to irredundant arrays), the separable (full factorial array), as well as several other cases with intermediate entanglement (arrays with any given number of redundancies). In this way, one obtains an approximate picture of the complete classification in $\mathcal H_d^{\otimes N}$. The number of classes based on orthogonal arrays is finite, whereas for quantum systems with $N>3$ parties there exist infinitely many classes \cite{VDMV02}.

Because the generating OAs have typically a smaller number of runs then generic arrays, it is natural, in the search for highly entangled states, to turn to the set $\mathcal G(N,d,k)$ corresponding to generating arrays with $N$ columns, $d$ letters and strength $k$. In fact, the set $\mathcal G(N,d,k)$ contains all $k$-uniform states of minimal support, as stated in the following proposition:

\begin{prop}\label{propMDS}
Every maximum distance separable code gives rise to a quantum state based on a generating OA in the Hilbert basis of $C(N,d,k)$, for some values of the parameters $N$, $d$ and $k$.  
\end{prop}
\begin{proof}
Every vector $|\Psi_{N,d,k}\rangle$ associated to an OA($r,N,d,k$) can be obtained as a non-negative integer combination of the generating states, i.e. $|\Psi_{N,d,k}\rangle=\sum_{i=1}^{m}\alpha_i\,|g^{(i)}_{N,d,k}\rangle$ with $\alpha_i \in \mathbb N_0$. If the OA associated to the $i$th generator $|g^{(i)}_{N,d,k}\rangle$ has index $\lambda^{(i)}_{N,d,k}$ then the OA associated to the state $|\Psi_{N,d,k}\rangle$ has index $\lambda=\sum_{i=i}^{m}\alpha_i\lambda^{(i)}_{N,d,k}$. Given that OAs associated to MDS codes have index unity, i.e. $\lambda=1$, and $\alpha_i\geq0$, we conclude that $|\Psi_{N,d,k}\rangle$ have to be a generator state $|g^{(i)}_{N,d,k}\rangle$, for a suitable value of $i\in[1,\dots,m]$.
\end{proof}

Let us observe, for example, that for the set $\mathcal G(2,2,1)$ we find the maximally entangled states given in Eq.~\eqref{gen21}, namely the Bell states.
 The set $\mathcal G(3,2,2)$ includes the states of Eq.~\eqref{g32}, 
which are both equivalent under local unitary operations to the GHZ state $\ket{GHZ_3}$
of three qubits. The $W$ class, inequivalent to the GHZ class, does not appear in $\mathcal G(3,2,k)$. This may be related to the fact that any 3-qubit state in the $W$ class 
has zero three-tangle \cite{CKW00}, while the GHZ class is characterized by a positive 
three-tangle. Thus, 3-qubit pure states belonging to the $W$ class form a measure zero set within the set of pure states, whereas the GHZ class is generic, i.e.~a random $3$--qubit pure quantum state belongs to the GHZ class with probability one. Despite this fact, the 3-qubit $W$ class is encoded in some elements of $\mathcal G(4,2,1)$. 
Indeed, the array
\be
\text{OA$_{4,2,1}^{(42)}$}=\begin{array}{cccc}
 0 & 0 & 0 & 1 \\
 0 & 0 & 1 & 0 \\
 0 & 1 & 0 & 0 \\
 1 & 0 & 0 & 0 \\
 1 & 1 & 1 & 1 \\
 1 & 1 & 1 & 1 \\
\end{array}
\ee
produces a three qubit $W$ state, $|W_3\rangle=|001\rangle+|010\rangle+|100\rangle$, 
if the first user -- Alice -- applies a projective measurement 
in the computational basis and obtains the outcome zero --
see Appendix \ref{app1}.

\section{Detection of non-isomorphic arrays with help of entanglement theory}\label{S5}
In this section, we show how the theory of quantum entanglement can be useful to 
determine whether two given orthogonal arrays are isomorphic. This problem is equivalent to determining whether their associated pure quantum states are equivalent under swap operations and local unitary operations corresponding to permutations of levels. 
The following result holds:
\begin{prop}\label{prop5}
Let $|\Psi_{N,d,k}\rangle$ and $|\tilde{\Psi}_{N,d,k}\rangle$ be two $N$-partite pure states associated by Eq.(\ref{OAPsi}) to two given arrays. 
A necessary condition for the arrays to be isomorphic
 is that for both sets of single party reductions 
\begin{equation}
\Gamma:\{\rho^{(i)}\,|\, \rho^{(i)}=\mathrm{Tr}_{1,\dots,i-1,i+1,\dots,N}|\Psi_{N,d,k}\rangle\langle\Psi_{N,d,k}|\}
\end{equation}
and
\begin{equation}
\tilde{\Gamma}:\{\tilde{\rho}^{(i)}\,|\, \tilde{\rho}^{(i)}=\mathrm{Tr}_{1,\dots,i-1,i+1,\dots,N}|\tilde{\Psi}_{N,k}\rangle\langle\tilde{\Psi}_{N,k}|\}\;,
\end{equation}
where $i=1,\dots, N$, there exists an injective permutation function $f:\mathbb{Z}_N\rightarrow\mathbb{Z}_N$ and a suitable permutation matrix $P_i$ such that
\begin{equation}\label{th}
\rho^{(i)}=P_i\,\tilde{\rho}^{(f[i])}P_i^{T}\;,
\end{equation}
where $\bullet^T$ denotes matrix transposition.
\end{prop}
\begin{proof}
Let $\ket{\Psi_{N,d,k}}$ be the pure quantum state corresponding to the orthogonal array OA($r,N,d,k$). One has to check that permutations of rows, columns and  symbols within columns, lead to quantum states $\ket{\tilde\Psi_{N,d,k}}$, such that the single party reductions $\tilde{\rho}^{(i)}$ are related to the reductions ${\rho}^{(i)}$ of $\ket{\Psi_{N,d,k}}$ through Eq.~\eqref{th}. Interchanges of rows do not modify the single party reductions since, under interchanges of rows only, $\ket{\tilde\Psi_{N,d,k}}$ coincides with $\ket{\Psi_{N,d,k}}$. Interchanges of columns correspond to a renaming of the labels of the local Hilbert spaces, which is taken into account by a suitable choice of the permutation function $f$. Permutations of symbols within columns are described by local operations of the form $P_1\otimes \dots\otimes P_N$, where $\{P_i\}_{i=1,\dots,N}$ are permutation matrices. Let us recall that, for any two states equivalent under local unitary transformations, $\ket{\phi_1}$ and $\ket{\phi_2}=U_1\otimes\dots\otimes U_N\,\ket{\phi_1}$, one has that the single party reductions are related as follows,
\be
\rho_2^{(i)}= U_i\, \rho_1^{(i)}\, U_i^\dagger\;,
\ee
where $\rho_1^{(i)}$ and $\rho_2^{(i)}$ denote the $i^{\text{th}}$ single party reduction of $\ket{\phi_1}$ and $\ket{\phi_2}$, respectively. Since any permutation matrix is also unitary, one arrives at Eq.~\eqref{th}. 
\end{proof}

\begin{table}
\begin{center}
\begin{tabular}{|C{1.5cm}|C{1.5cm}|C{1.5cm}|C{1.5cm}|C{1.5cm}|C{1.5cm}|C{1.5cm}|}
\hline
$ $ & $I_1$&$I_2$&$I_3$&$I_4$&$I_5$&$I_6$\\
\hline
$|\Sigma_1\rangle$ & 8 &13/18&13/18&5/9&13/36&0\\
$|\Sigma_2\rangle$ & 8 &7/9&5/9&5/9&1/3&4/81\\
\hline
\end{tabular}
\end{center}
\caption{Polynomial invariants $I_1, \dots, I_6$  for the three qubit states $\ket{\Sigma_1}$ and $\ket{\Sigma_2}$, which correspond to the  arrays $OA_1$ and $OA_2$.
 We follow the conventions of Sudbery \cite{sudbery2001local}.}
\label{Table3}
\end{table}

It is easy to show that the conditions imposed by Proposition \ref{prop5} are necessary, but not sufficient. For example, there are two inequivalent 1-uniform quantum states for four qubits systems:
\begin{equation}\label{quadr}
|\Xi_1\rangle=|0000\rangle+|0011\rangle+\ket{1101}+\ket{1110},
\end{equation}
and
\begin{equation}\label{bip}
|\Xi_2\rangle=|0000\rangle+|0011\rangle+|1100\rangle+|1111\rangle.
\end{equation}
Given that both states are 1-uniform, all single party reductions are maximally mixed. However, the corresponding OAs,
\be
OA_1=\ba{cccc}
0&0&0&0\\
0&0&1&1\\
1&1&0&1\\
1&1&1&0
\ea\;\qquad\mbox{and}\qquad
OA_2=\ba{cccc}
0&0&0&0\\
0&0&1&1\\
1&1&0&0\\
1&1&1&1
\ea\;,
\ee
are non-isomorphic.

We can use Proposition \ref{prop5} to define a sufficient criterion to detect non-isomorphic orthogonal arrays.
\begin{corol}\label{cor1}
Let $|\Psi_{N,d,k}\rangle$ and $|\tilde{\Psi}_{N,d,k}\rangle$ be $N$-partite pure states associated to two given OA($r,N,d,k$) and let 
\begin{equation}
\{\pi^{(i)}=\mathrm{Tr}[(\rho^{(i)})^2]\}\quad\mbox{ and }\quad\{\tilde{\pi}^{(i)}=\mathrm{Tr}[(\tilde{\rho}^{(i)})^2]\}
\end{equation}
be the sets of purity of reductions $\rho^{(i)}_j$. If the sets $\{\pi^{(i)}\}$ and $\{\tilde{\pi}^{(i)}\}$ are not identical then the OAs are non-isomorphic. 
\end{corol}
\begin{proof}
The proof follows immediately from the previous proposition, together with the observation that any two matrices $\rho_1$ and $\rho_2$, which are connected by a similarity transformation of the form $\rho_1=P\, \rho_2P^T$ with $P$ a permutation matrix, have the same purity. 
\end{proof}

Another sufficient criterion is the following:
\begin{prop}\label{prop6}
If two $N$-partite pure quantum states $|\Psi_{N,d,k}\rangle$ and $|\tilde{\Psi}_{N,d,k}\rangle$ have different amounts of entanglement, according to any measure, then the associated arrays  defined by Eq.(\ref{OAPsi}) are non-isomorphic.
\end{prop} 
The proof follows straightforwardly from the fact that both the swap operations and local permutation operations preserve entanglement. As an example, consider the following two orthogonal arrays,
\begin{align*}
OA_1=\begin{array}{ccc}
 0 & 0 & 0 \\
 0 & 0 & 0 \\
 0 & 0 & 1 \\
 0 & 1 & 0 \\
 1 & 0 & 1 \\
 1 & 1 & 0 \\
 1 & 1 & 1 \\
 1 & 1 & 1 \\
\end{array}\;,\qquad\qquad\qquad\qquad
OA_2=\begin{array}{ccc}
 0 & 0 & 0 \\
 0 & 0 & 0 \\
 0 & 0 & 1 \\
 0 & 1 & 1 \\
 1 & 0 & 1 \\
 1 & 1 & 0 \\
 1 & 1 & 0 \\
 1 & 1 & 1 \\
\end{array}\;,
\end{align*}
which give rise to the 3-qubit states $\ket{\Sigma_1}$ and $\ket{\Sigma_2}$, respectively. Since the polynomial invariants \cite{sudbery2001local} for $\ket{\Sigma_1}$ and $\ket{\Sigma_2}$ are not the same, $OA_1$ and $OA_2$ are non-isomorphic arrays (see Table \ref{Table3}).

\section{Entanglement classes of generating states}
\label{S6}
In this section, we describe the entanglement classes of array-based states belonging to the set $\mathcal G(N,d,k)$, for some small values of the parameters $N$, $d$ and $k$. 
This involves two separate steps: \emph{(i)} computing the Hilbert basis of the cone $C(N,d,k)$, that is finding the set $\mathscr G(N,d,k)$ of generating arrays; \emph{(ii)} identifying the non-isomorphic arrays within $\mathscr G(N,d,k)$, which are the representative states of each class.

The first task is a standard integer programming task. Concerning the second task, we make use of the fact that any orthogonal array can be encoded as a vertex-colored graph, in such a way that two arrays are isomorphic iff the corresponding graphs are \cite{bulutoglu2008classification}. 
Testing for graph isomorphism was performed via the software program \texttt{nauty} \cite{mckay2014practical}. 

\subsection{Three qubits}
Let us consider in detail the case $N=3$, $k=1$; since $\mathscr{O\!A}(N,d,k)\subset \mathscr{O\!A}(N,d,k')$ whenever $k'<k$, we assume the lowest possible value of strength. There are only two non-isomorphic arrays in the set $\mathscr G(3,2,1)$:
\be\label{oanik1}
\text{OA$_{3,2,1}^{(4)}$} = 
\begin{matrix}
0 & 0 & 0 \\
1 & 1 & 1 
\end{matrix}\;, \qquad \qquad
\text{OA$_{3,2,1}^{(6)}$} = 
\begin{matrix}
0 & 0 & 0 \\
0 & 1 & 1 \\
1 & 0 & 1 \\
1 & 1 & 0 
\end{matrix}\;.
\ee
Therefore, there are only two classes in the set $\mathcal G(3,2,1)$ of generating states, 
with representatives denoted by $\ket{\phi_{3,2,1}^{(\gamma)}}$,
\be
\begin{split}
\ket{\phi_{3,2,1}^{(\rn{1})}} =& \ket{g_{3,2,1}^{(4)}} = \ket{000} + \ket{111}\;,\\
\ket{\phi_{3,2,1}^{(\rn{2})}} =& \ket{g_{3,2,1}^{(6)}} = \ket{000} + \ket{011} +\ket{101} + \ket{110}\;. 
\end{split}
\ee

This case is sufficiently simple that we also show the entanglement classes for $\mathcal{O\!A}(3,2,1)$. One finds that there are 9 classes, with representative states $\ket{\psi_{3,2,1}^{(\gamma)}}$:
\be
\begin{split}
\ket{\psi_{3,2,1}^{(\rn{1})}}\; =\; &\ket{000}+\ket{111}\;,\\
\ket{\psi_{3,2,1}^{(\rn{2})}}\; =\; &\ket{000}+\ket{011}+\ket{101}+\ket{110}\;,\\
\ket{\psi_{3,2,1}^{(\rn{3})}}\; =\; &\, 2\ket{000}+\ket{001}+\ket{110}+2\ket{111}\;,\\
\ket{\psi_{3,2,1}^{(\rn{4})}}\; =\; &\, 2\ket{000}+\ket{011}+\ket{101}+\ket{110}+\ket{111}\;,\\
\ket{\psi_{3,2,1}^{(\rn{5})}}\; =\; &\, 3 \ket{000}+\ket{011}+\ket{101}+\ket{110}+2\ket{111}\;,\\
\ket{\psi_{3,2,1}^{(\rn{6})}}\; =\; &\, 2\ket{000}+\ket{001}+\ket{011}+\ket{101}+2\ket{110}+\ket{111}\;,\\
\ket{\psi_{3,2,1}^{(\rn{7})}}\; =\; &\, 2\ket{000}+\ket{001}+\ket{010}+\ket{101}+\ket{110}+2\ket{111}\;,\\
\ket{\psi_{3,2,1}^{(\rn{8})}}=\; & \ \ket{000}+\ket{001}+\ket{110}+\ket{111}\;,\\
\ket{\psi_{3,2,1}^{(\rn{9})}}\; =\; & \ \ket{000}+\ket{001}+\ket{010}+\ket{011}+\ket{100}+\ket{101}+\ket{110}+\ket{111}\;.
\end{split}
\ee 
The first two states are based on the OAs of Eq.~\eqref{oanik1}; they are unitarily equivalent to $\ket{GHZ_3}$. The states $\ket{\psi_{3,2,1}^{(\gamma)}}$ (with $\gamma= \rn{3},\dots,\rn{6}$) are of the GHZ type, but not unitarily equivalent to $\ket{GHZ_3}$. $\ket{\psi_{3,2,1}^{(\rn{7})}}$ is of the W type. 
The state
$\ket{\psi_{3,2,1}^{(\rn{8})}}$ has only bipartite entanglement among the first two parties; in fact, it can be written as $\ket{GHZ_2}\otimes \ket{+}$. 
The state $\ket{\psi_{3,2,1}^{(\rn{9})}}$ is fully separable and corresponds to the full factorial $\mathfrak F_{3,2}$ --
 see also Table \ref{Table1}. We have therefore been able to reproduce all qualitatively different types of entanglement for three qubits \cite{dur2000three}. Let us remark that states that are inequivalent in the array-based classification can nonetheless be equivalent in $\mathcal H_d^{\otimes N}$, in which the set of free operations  is larger.

\begin{table}
\begin{center}
\begin{tabular}{|C{1.6cm}|C{1.6cm}|C{1.6cm}|C{1.6cm}|C{1.6cm}|C{1.6cm}|C{1.6cm}|}
\hline
$\ket{\psi_{3,2,1}^{(\gamma)}}$ & $I_2$ & $I_3$ & $I_4$ & $I_5$ & $I_6$ & type \\[0.15cm]
\hline
$\ket{\psi_{3,2,1}^{(\rn{1})}}$ & 1/2 & 1/2 & 1/2 & 1/4 & 1/4 & GHZ \\ [0.15cm]

$\ket{\psi_{3,2,1}^{(\rn{2})}}$ & 1/2 & 1/2 & 1/2 & 1/4 & 1/4 & GHZ \\ [0.15cm]

$\ket{\psi_{3,2,1}^{(\rn{3})}}$  & 41/50 & 1/2 & 1/2 & 1/4 & 81/2500 & GHZ\\ [0.15cm]

$\ket{\psi_{3,2,1}^{(\rn{4})}}$  & 9/16 & 9/16 & 9/16 & 73/256 & 9/64 & GHZ\\ [0.15cm]

$\ket{\psi_{3,2,1}^{(\rn{5})}}$  &  9/16 & 9/16 & 9/16 & 73/256 & 9/64 &  GHZ\\ [0.15cm]

$\ket{\psi_{3,2,1}^{(\rn{6})}}$  & 7/9 & 5/9 & 5/9 & 1/3 & 4/81 & GHZ\\ [0.15cm]

$\ket{\psi_{3,2,1}^{(\rn{7})}}$  & 13/18 & 13/18 & 5/9 & 13/36 & 0 &  W \\ [0.15cm]

$\ket{\psi_{3,2,1}^{(\rn{8})}}$  &  1& 1/2 & 1/2 & 1/4 & 0 &  bipartite \\ [0.15cm]

$\ket{\psi_{3,2,1}^{(\rn{9})}}$  & 1 & 1 & 1 & 1 & 0 &  separable \\ [0.15cm]
\hline
\end{tabular}
\end{center}
\caption{Representative $3$--qubit states for the nine entanglement classes in $\mathcal{O\!A}(3,2,1)$ and their polynomial invariants. We follow the convention of Sudbery \cite{sudbery2001local}. In particular, $I_6$ is related to the three-tangle $\tau_{ABC}$ via $I_6 = \tau_{ABC}^2/4$. We omit $I_1$, which is trivially the norm of the state. States having the same invariants are unitarily equivalent in $\mathcal H_2^{\otimes 3}$.}
\label{Table1}
\end{table}

\subsection{Four qubits}
Let us now consider the case of 4 qubits. For strength $k=1$, there are only three
 non-isomorphic arrays in $\mathscr G(4,2,1)$: OA$_{4,2,1}^{(8)}$, OA$_{4,2,1}^{(32)}$ and OA$_{4,2,1}^{(48)}$; they are shown explicitly in Appendix ~\ref{app1}. Hence
 they lead to three corresponding classes of entanglement, with representative states:
\be
\begin{split}
\ket{\phi_{4,2,1}^{(\rn{1})}}=&\ket{0000}+\ket{1111}\;,\\
\ket{\phi_{4,2,1}^{(\rn{2})}}=&\ket{0000}+\ket{0011}+\ket{1101}+\ket{1110}\;, \\
\ket{\phi_{4,2,1}^{(\rn{3})}}=&\,2\ket{0000}+\ket{0111}+\ket{1011}+\ket{1101}+\ket{1110}\;. \\
\end{split}
\ee 
The first one is the GHZ state $\ket{GHZ_4}$. The second is also a 1-uniform state, like $\ket{GHZ_4}$, but not SLOCC-equivalent to it. The third has genuine quadripartite entanglement, with average purity of its single-party reductions equal to $\sim 0.531$.
For more information see Table \ref{Table2}. 

Finding all classes in $\mathcal{O\!A}(4,2,1)$ is significantly more complicated than for 3 qubits; there are in fact a total of 1110 different classes. It is reasonable to expect the number of classes to grow quickly with the number of parties $N$, though by construction it always remains finite. 

\begin{table}
\begin{center}
\begin{tabular}{|C{1.6cm}|C{1.6cm}|C{1.6cm}|C{1.6cm}|C{1.6cm}|}
\hline
$\ket{\phi_{4,2,1}^{(\gamma)}}$ & $H$ & $L$ & $M$ & $D$  \\[0.15cm]
\hline
$\ket{\phi_{4,2,1}^{(\rn{1})}}$ & 1/2 & 0 & 0 & 0  \\ [0.15cm]

$\ket{\phi_{4,2,1}^{(\rn{2})}}$ & 0 & 0 & 1/16 & 0  \\ [0.15cm]

$\ket{\phi_{4,2,1}^{(\rn{3})}}$ & 0 & 0 & 0 & 1/128  \\ [0.15cm]
\hline
\end{tabular}
\end{center}
\caption{Representative $4$--qubit states for the three entanglement classes in the set $\mathcal{G}(4,2,1)$ and their polynomial invariants. We follow here the convention
of Luque and Thibon \cite{LT03}. 
States having the same invariants are SLOCC-equivalent in $\mathcal H_2^{\otimes 4}$.}
\label{Table2}
\end{table}

We turn instead to a more detailed study of the entanglement properties of states in $\mathcal G(4,2,k)$.
The hyperdeterminant is a generalization of the determinant of a matrix to tensors \cite{GKZ94}. Its absolute value is used in quantum information theory as an entanglement measure, generalizing the concurrence \cite{HW97} and three-tangle $\tau_{ABC}$ for systems of two and three qubits, respectively \cite{MW02}. 
In Appendix \ref{App2}, we compute the hyperdeterminant, as well as other entanglement measures \cite{LT03}, for every state in $\mathcal{G}(4,2,k)$ and all 
possible values of the strength $k$.

Let us remark that, for 4-qubit systems, there are nine continuous families of entangled states in the SLOCC classification \cite{VDMV02}. It has been proven that eight of those classes have zero hyperdeterminant \cite{LT03}, the remaining one being called \emph{generic}. 
Up to SLOCC operations any state belonging to the generic class can be written as \cite{GW10},
\begin{equation}\label{generic_class}
|\phi\rangle=\sum_{j=0}^3 a_j|u_j\rangle,
\end{equation}
where $|u_0\rangle=|\Phi^+\rangle|\Phi^+\rangle$, $|u_1\rangle=|\Phi^-\rangle|\Phi^-\rangle$, $|u_2\rangle=|\Psi^+\rangle|\Psi^+\rangle$, $|u_3\rangle=|\Psi^-\rangle|\Psi^-\rangle$ and $|\Phi^{\pm}\rangle=|00\rangle\pm|11\rangle$ and $|\Psi^{\pm}\rangle=|01\rangle\pm|10\rangle$ are the Bell states. We observe that any 4-qubit state $|\phi\rangle$ belonging to the generic class can be written as a linear combination of the states associated to the following four elements of the Hilbert basis for the class $\mathscr{O\!A}(4,2,1)$:
\begin{center}
OA$_{4,2,1}^{(2)}$,\quad OA$_{4,2,1}^{(3)}$, \quad OA$_{4,2,1}^{(5)}$, \quad OA$_{4,2,1}^{(8)}$.
\end{center}
Explicit expressions for these OAs can be found in Appendix \ref{app1}. Furthermore, elements OA$_{4,2,1}^{(j)},\,j\in\{1,32\}$ and OA$_{4,2,1}^{(j)},\,j\in\{33,48\}$ give rise to states with hyperdeterminants $\Delta=0$ and $\Delta=-0.506821$, respectively. On the other hand, for the class $\mathscr{O\!A}(4,2,2)$, elements of the Hilbert basis OA$_{4,2,2}^{(j)},\,j\in\{1,10\}$ and OA$_{4,2,2}^{(j)},\,j\in\{11,26\}$ have hyperdeterminants $\Delta=0$ and $\Delta=-0.013995$, respectively. Finally, the two generating OAs of $\mathscr{O\!A}(4,2,3)$ have hyperdeterminant $\Delta=0$.

\subsection{Five qubits and beyond}
Let us now consider systems consisting of five qubits. For strength $k=1$, there are eleven non-isomorphic generating arrays in $\mathscr G(5,2,1)$, giving rise to as many entanglement classes in the 
corresponding set $\mathcal G(5,2,1)$. Here, we report a possible choice of representative states:
\begin{align*}
\ket{\phi_{5,2,1}^{(\rn{1})}}=& \ket{00000}+\ket{11111}\;,\\
\ket{\phi_{5,2,1}^{(\rn{2})}}=& \ket{00111}+\ket{01000}+\ket{10100}+\ket{11011} \;,\\
\ket{\phi_{5,2,1}^{(\rn{3})}}=& \ket{00111}+\ket{01000}+\ket{10011}+\ket{11100}\;,\\
\ket{\phi_{5,2,1}^{(\rn{4})}}=& \ket{00100}+\ket{00001}+\ket{00010}+\ket{11000}+2\ket{11111} \;,\\
\ket{\phi_{5,2,1}^{(\rn{5})}}=& \ket{00000}+\ket{00110}+\ket{01001}+\ket{10011}+\ket{11101}+\ket{11110} \;,\\
\ket{\phi_{5,2,1}^{(\rn{6})}}=& \ket{00100}+\ket{01000}+\ket{00011}+\ket{10011}+\ket{11101}+\ket{11110} \;,\tag{\stepcounter{equation}\theequation}\\
\ket{\phi_{5,2,1}^{(\rn{7})}}=& \ket{01000}+\ket{00011}+\ket{00101}+\ket{00110}+\ket{11011}+\ket{11101}+\ket{11110}+\ket{10000} \;,\\
\ket{\phi_{5,2,1}^{(\rn{8})}}=&\,\ket{00000}+2\ket{01111}+\ket{10101}+\ket{10110}+\ket{11000}+\ket{11011} \;,\\
\ket{\phi_{5,2,1}^{(\rn{9})}}=& \ket{01000}+\ket{00000}+2\ket{00111}+\ket{11011}+\ket{11101}+\ket{11110}+\ket{10000}  \;,\\
\ket{\phi_{5,2,1}^{(\rn{10})}}=&\ket{01000}+\ket{00100}+\ket{00001}+\ket{00010}+\ket{10000}+3\ket{11111}  \;,\\
\ket{\phi_{5,2,1}^{(\rn{11})}}=&\, 3\ket{00000}+2\ket{01111}+\ket{11100}+\ket{11001}+\ket{11010}+2\ket{10111} \;.
\end{align*}

For six qubits, the problem of computing the Hilbert basis of the cone $C(6,2,k)$ is computationally out of reach for small values of $k$. In fact, we managed to obtain the generating arrays only for $k=4,\,5$. Here, the case $k=4$ is discussed, since the case $k=5$ is already included in it. There are three non-isomorphic generating arrays in $\mathscr G(6,2,4)$, giving rise to as many classes, with the following representative states:  
\begin{align*}
\ket{\phi_{6,2,4}^{(\rn{1})}}=&
\ket{ 0  0  0  0  0  0}+ 
\ket{ 0  0  0  0  0  1}+ 
\ket{ 0  0  0  1  1  0}+ 
\ket{ 0  0  0  1  1  1}+
\ket{ 0  0  1  0  1  0}+
\ket{ 0  0  1  0  1  1}+
\ket{ 0  0  1  1  0  0}\\
&+\ket{ 0  0  1  1  0  1}+
\ket{ 0  1  0  0  1  0}+
\ket{ 0  1  0  0  1  1}+
\ket{ 0  1  0  1  0  0}+
\ket{ 0  1  0  1  0  1}+
\ket{ 0  1  1  0  0  0}+
\ket{ 0  1  1  0  0  1}\\
&+\ket{ 0  1  1  1  1  0}+
\ket{ 0  1  1  1  1  1}+
\ket{ 1  0  0  0  1  0}+
\ket{ 1  0  0  0  1  1}+
\ket{ 1  0  0  1  0  0}+
\ket{ 1  0  0  1  0  1}+
\ket{ 1  0  1  0  0  0}\\
&+\ket{ 1  0  1  0  0  1}+
\ket{ 1  0  1  1  1  0}+
\ket{ 1  0  1  1  1  1}+
\ket{ 1  1  0  0  0  0}+
\ket{ 1  1  0  0  0  1}+
\ket{ 1  1  0  1  1  0}+
\ket{ 1  1  0  1  1  1}\\
&+\ket{ 1  1  1  0  1  0}+
\ket{ 1  1  1  0  1  1}+
\ket{ 1  1  1  1  0  0}+
\ket{ 1  1  1  1  0  1}\;,\\
\ket{\phi_{6,2,4}^{(\rn{2})}}=&
\ket{0  0  0  0  0  1}+
\ket{0  0  0  0  1  0}+
\ket{0  0  0  1  0  0}+
\ket{0  0  0  1  1  1}+
\ket{0  0  1  0  0  0}+
\ket{0  0  1  0  1  1}+
\ket{0  0  1  1  0  1}\\
&+\ket{0  0  1  1  1  0}+
\ket{0  1  0  0  0  0}+
\ket{0  1  0  0  1  1}+
\ket{0  1  0  1  0  1}+
\ket{0  1  0  1  1  0}+
\ket{0  1  1  0  0  1}+
\ket{0  1  1  0  1  0}\\
&+\ket{0  1  1  1  0  0}+
\ket{0  1  1  1  1  1}+
\ket{1  0  0  0  0  0}+
\ket{1  0  0  0  1  1}+
\ket{1  0  0  1  0  1}+
\ket{1  0  0  1  1  0}+
\ket{1  0  1  0  0  1}\\
&+\ket{1  0  1  0  1  0}+
\ket{1  0  1  1  0  0}+
\ket{1  0  1  1  1  1}+
\ket{1  1  0  0  0  1}+
\ket{1  1  0  0  1  0}+
\ket{1  1  0  1  0  0}+
\ket{1  1  0  1  1  1}\\
&+\ket{1  1  1  0  0  0}+
\ket{1  1  1  0  1  1}+
\ket{1  1  1  1  0  1}+
\ket{1  1  1  1  1  0}\;,\\
\ket{\phi_{6,2,4}^{(\rn{3})}}=&
2\ket{0  0  0  0  0  0}+
\ket{0  0  0  0  0  1}+
\ket{0  0  0  0  1  0}+
\ket{0  0  0  0  1  1}+
\ket{0  0  0  1  0  0}+
\ket{0  0  0  1  0  1}+
\ket{0  0  0  1  1  0}\\
&+2\ket{0  0  0  1  1  1}+
2\ket{0  0  1  0  0  1}+
2\ket{0  0  1  0  1  0}+
\ket{0  0  1  0  1  1}+
2\ket{0  0  1  1  0  0}+
\ket{0  0  1  1  0  1}+
\ket{0  0  1  1  1  0}\\
&+\ket{0  0  1  1  1  1}+
\ket{0  1  0  0  0  0}+
\ket{0  1  0  0  0  1}+
\ket{0  1  0  0  1  0}+
2\ket{0  1  0  0  1  1}+
\ket{0  1  0  1  0  0}+
2\ket{0  1  0  1  0  1}\\
&+2\ket{0  1  0  1  1  0}+
2\ket{0  1  1  0  0  0}+
\ket{0  1  1  0  0  1}+
\ket{0  1  1  0  1  0}+
\ket{0  1  1  0  1  1}+
\ket{0  1  1  1  0  0}+
\ket{0  1  1  1  0  1}\\
&+\ket{0  1  1  1  1  0}+
2\ket{0  1  1  1  1  1}+
\ket{1  0  0  0  0  0}+
\ket{1  0  0  0  0  1}+
\ket{1  0  0  0  1  0}+
2\ket{1  0  0  0  1  1}+
\ket{1  0  0  1  0  0}\\
&+2\ket{1  0  0  1  0  1}+
2\ket{1  0  0  1  1  0}+
2\ket{1  0  1  0  0  0}+
\ket{1  0  1  0  0  1}+
\ket{1  0  1  0  1  0}+
\ket{1  0  1  0  1  1}+
\ket{1  0  1  1  0  0}\\
&+\ket{1  0  1  1  0  1}+
\ket{1  0  1  1  1  0}+
2\ket{1  0  1  1  1  1}+
\ket{1  1  0  0  0  0}+
2\ket{1  1  0  0  0  1}+
2\ket{1  1  0  0  1  0}+
2\ket{1  1  0  1  0  0}\\
&+3\ket{1  1  0  1  1  1}+
\ket{1  1  1  0  0  0}+
\ket{1  1  1  0  0  1}+
\ket{1  1  1  0  1  0}+
2\ket{1  1  1  0  1  1}+
\ket{1  1  1  1  0  0}+
2\ket{1  1  1  1  0  1}\\
&+2\ket{1  1  1  1  1  0}\;.\tag{\stepcounter{equation}\theequation}
\end{align*}

To conclude this section, we remark that the combinatorial techniques presented in this work are not restricted to $N$-qudit systems. That is, they can also be applied to heterogeneous systems \cite{goyeneche2016multipartite}, made up of subsystems with a different number of internal levels. We have computed the Hilbert basis for the families \emph{(i)} $\mathscr{O\!A}(2^23^1,k)$ and \emph{(i)} $\mathscr{O\!A}(2^13^2,k)$, with $k=1,2$. Case \emph{(i)} corresponds to the first two columns having $d_1=2$ symbols and the third column having $d_2=3$ symbols, with $k$ denoting the strength; instead, case \emph{(ii)} corresponds to the first column having $d_1=2$ symbols and the second and third columns having $d_2=3$ symbols. The complete Hilbert bases can be found in the supplementary material provided online \cite{HBonline}. Here, we report the representative states for the entanglement classes of $\mathcal G(2^23^1, k)$, with strength $k=1$. For systems of two qubits and one qutrit, $\mathcal G(2^23^1, 1)$ is made up of six classes, with representative states:
\be
\begin{split}
\ket{\phi_{2^23^1,1}^{(\rn{1})}}=&\ket{0  0  0}+
\ket{1  1  0}+
\ket{0  0  1}+
\ket{0  1  2}+
\ket{1  0  2}+
\ket{1  1  1}\;,\\
\ket{\phi_{2^23^1,1}^{(\rn{2})}}=&\ket{0  0  0}+
\ket{0  1  0}+
\ket{1  0  1}+
\ket{1  1  2}+
\ket{1  1  1}+
\ket{0  0  2}\;,\\
\ket{\phi_{2^23^1,1}^{(\rn{3})}}=&2\ket{0  0  0}+
\ket{0  1  1}+
\ket{1  0  1}+
\ket{1  1  2}+
\ket{1  1  2}\;,\\
\ket{\phi_{2^23^1,1}^{(\rn{4})}}=&\ket{0  0  0}+
\ket{1  1  0}+
2\ket{0  0  1}+
\ket{1  1  2}+
\ket{1  1  2}\;,\\
\end{split}
\ee
\begin{equation*}
\begin{split}
\ket{\phi_{2^23^1,1}^{(\rn{5})}}=&\ket{0  0  0}+
\ket{1  1  0}+
\ket{0  0  1}+
\ket{0  0  2}+
\ket{1  1  1}+
\ket{1  1  2}\;,\\
\ket{\phi_{2^23^1,1}^{(\rn{6})}}=&2\ket{0  0  0}+
\ket{0  1  1}+
\ket{1  0  2}+
\ket{1  1  2}+
\ket{1  1  1}\;.
\end{split}
\end{equation*}
\vspace{1pt}

Instead, for systems of one qubit and two qutrits, one finds fifteen classes, with the following representatives:
\begin{align*}
\ket{\phi_{2^13^2,1}^{(\rn{1})}}=&\ket{0  0  0}+
\ket{1  1  2}+
\ket{1  2  1}+
\ket{1  0  0}+
\ket{0  1  2}+
\ket{0  2  1}\;,\\
\ket{\phi_{2^13^2,1}^{(\rn{2})}}=&\ket{0  0  0}+
\ket{1  1  1}+
\ket{1  0  2}+
\ket{1  2  0}+
\ket{0  1  2}+
\ket{0  2  1}\;,\\
\ket{\phi_{2^13^2,1}^{(\rn{3})}}=&2\ket{0  0  0}+
\ket{1  1  1}+
\ket{1  1  2}+
\ket{1  2  1}+
\ket{0  2  2}\;,\\
\ket{\phi_{2^13^2,1}^{(\rn{4})}}=&\ket{0  0  0}+
\ket{1  2  1}+
\ket{1  1  2}+
\ket{1  1  0}+
\ket{0  0  2}+
\ket{0  2  1}\;,\\
\ket{\phi_{2^13^2,1}^{(\rn{5})}}=&\ket{0  0  0}+
\ket{1  2  0}+
\ket{1  1  2}+
\ket{1  1  1}+
\ket{0  0  2}+
\ket{0  2  1}\;,\\
\ket{\phi_{2^13^2,1}^{(\rn{6})}}=&2\ket{0  0  0}+
2\ket{1  1  1}+
\ket{1  2  2}+
\ket{0  2  2}\;,\\
\ket{\phi_{2^13^2,1}^{(\rn{7})}}=&\ket{0  0  0}+
\ket{1  1  1}+
\ket{1  0  2}+
\ket{1  2  0}+
\ket{0  1  1}+
\ket{0  2  2}\;,\\
\ket{\phi_{2^13^2,1}^{(\rn{8})}}=&\ket{0  0  0}+
\ket{1  1  2}+
\ket{1  1  1}+
\ket{1  0  1}+
\ket{0  2  0}+
\ket{0  2  2}\;,\tag{\stepcounter{equation}\theequation}\\
\ket{\phi_{2^13^2,1}^{(\rn{9})}}=&\ket{0  0  0}+
\ket{1  1  1}+
\ket{1  1  2}+
\ket{1  2  0}+
3\ket{1  0  1}+
2\ket{0  1  0}+
3\ket{0  2  2}\;,\\
\ket{\phi_{2^13^2,1}^{(\rn{10})}}=&\ket{0  0  0}+
3\ket{1  1  1}+
3\ket{1  0  2}+
3\ket{0  2  0}+
\ket{0  2  1}+
\ket{0  1  2}\;,\\
\ket{\phi_{2^13^2,1}^{(\rn{11})}}=&4\ket{0  0  0}+
3\ket{1  1  1}+
3\ket{1  2  2}+
\ket{0  1  2}+
\ket{0  2  1}\;,\\
\ket{\phi_{2^13^2,1}^{(\rn{12})}}=&3\ket{0  0  0}+
\ket{1  2  0}+
\ket{1  0  1}+
2\ket{1  1  1}+
2\ket{1  1  2}+
\ket{0  2  1}+
2\ket{0  2  2}\;,\\
\ket{\phi_{2^13^2,1}^{(\rn{13})}}=&\ket{0  0  0}+
\ket{1  1  2}+
\ket{1  2  1}+
2\ket{1  1  0}+
2\ket{1  0  1}+
\ket{0  1  1}+
2\ket{0  2  2}+
\ket{0  0  2}+
\ket{0  2  0}\;,\\
\ket{\phi_{2^13^2,1}^{(\rn{14})}}=&2\ket{0  0  0}+
\ket{1  1  2}+
\ket{1  2  1}+
2\ket{1  1  0}+
2\ket{1  0  1}+
\ket{0  1  1}+
3\ket{0  2  2}\;,\\
\ket{\phi_{2^13^2,1}^{(\rn{15})}}=&\ket{0  0  0}+
2\ket{1  1  2}+
2\ket{1  2  1}+
\ket{1  1  0}+
\ket{1  0  1}+
\ket{0  1  1}+
2\ket{0  0  2}+
2\ket{0  2  0}\;.
\end{align*}

\section{Concluding remarks}\label{S7}
We investigated the combinatorial family  %$\mathcal{O\!A}(N,d,k)$ 
of quantum states composed of $N$ subsystems, with $d$ levels each,
which are based on fractional orthogonal designs consisting of $N$ columns
of symbols from a $d$-letter alphabet.
In particular, we formulated the entanglement classification problem 
within this restricted set and studied the entanglement classes corresponding 
to array-based states for a small number of parties. 

This approach provides a coarse-grained picture of the full classification of entanglement,
 as shown explicitly for the 3-qubit system. 
For a higher number of parties, we restricted our attention to 
the subfamily of states corresponding to generating orthogonal arrays.
In this way we identified several classes of highly entangled states, 
including all states based on maximum distance separable codes. 

The list of generating orthogonal arrays for systems of up to four qubits is provided 
in  Appendix \ref{app1}, for all possible values of the strength $k$. 
The methods described in this work are also applicable to heterogeneous systems,
made up of subsystems with different numbers of levels.
 The list of the generating arrays for systems of two qubits and one qutrit 
or one qubit and two qutrits, as well as for 5 qubits, is available online \cite{HBonline}.

Finally, we pose a list of open questions: \emph{(i)} Prove, or disprove, Conjecture \ref{conjecture}, concerning the structure of the Hilbert basis of the rational cone
for a class of orthogonal arrays with arbitrary number of columns.
 \emph{(ii)} Find a necessary and sufficient criterion based on 
entanglement theory to distinguish non-isomorphic classes of orthogonal arrays. 
\emph{(iii)} Generalize the method presented here to find the generating elements for quantum orthogonal arrays, introduced in Ref.~\cite{goyeneche2017entanglement}.

\section*{Acknowledgements}
We are grateful to Charles Colbourn, Robert Craigen and Jennifer Seberry for discussions on
 orthogonal arrays. DG and K\.{Z}  acknowledge support from Narodowe Centrum Nauki under grant number DEC-2015/18/A/ST2/00274. LS acknowledges support by EU through the collaborative H2020 project QuProCS (Grant Agreement 641277) and by UniMI through the H2020 Transition Grant. LS thanks DG and K\.{Z}  for the very kind hospitality during his stay, under the support of an Erasmus+ traineeship program, at the Jagiellonian University, where this work was conceived.

\appendix

\section{Hilbert bases for two to four qubits systems}\label{app1}
\subsection{Two qubits}
For a two-qubit system, the set of orthogonal arrays $\mathscr{O\!A}(2,2,1)$ is in one-to-one correspondence with the polytope $P(2,2,1)$. The generating states of the Hilbert basis of the cone $C(2,2,1)$ were reported in Eq.~\eqref{gen21}. They are the well-known Bell states and are repeated here for completeness, i.e.
\begin{equation}
\begin{split}
&|g_{2,2,1}^{(1)}\rangle=|00\rangle+|11\rangle\;,\qquad |g_{2,2,1}^{(2)}\rangle=|01\rangle+|10\rangle\;.
\end{split}
\end{equation}
Notice that we work with unnormalized states for simplicity. The corresponding orthogonal arrays are
\be
\begin{split}
&\text{OA$_{2,2,1}^{(1)}$}=\ba{cc}
0&0\\
1&1
\ea\;,\qquad\qquad \text{OA$_{2,2,1}^{(2)}$}=\ba{cc}
0&1\\
1&0
\ea\;.
\end{split}
\ee

\subsection{Three qubits}
For a 3-qubit system, the set of orthogonal arrays $\mathscr{O\!A}(3,2,1)$ is in one-to-one correspondence with the polytope $P(3,2,1)$. The generating states of the Hilbert basis of the cone $C(3,2,1)$ are:
\be\label{g31a1}
\begin{gathered}
\ket{g_{3,2,1}^{(1)}}=\ket{011}+\ket{100}\;,\qquad \ket{g_{3,2,1}^{(2)}}=\ket{010}+\ket{101}\;,\\
\ket{g_{3,2,1}^{(3)}}=\ket{001}+\ket{110}\;,\qquad \ket{g_{3,2,1}^{(4)}}=\ket{000}+\ket{111}\;,\\
\ket{g_{3,2,1}^{(5)}}=\ket{001}+\ket{010}+\ket{100}+\ket{111}\;,\qquad \ket{g_{3,2,1}^{(6)}}=\ket{000}+\ket{011}+\ket{101}+\ket{110}\;.\\
\end{gathered}
\ee
The corresponding orthogonal arrays are listed below.
\begin{align*}
\text{OA$_{3,2,1}^{(1)}$}=\begin{array}{ccc}
 0 & 1 & 1 \\
 1 & 0 & 0 \\
\end{array}\;,\qquad
\text{OA$^{(2)}_{3,2,1}$}=\begin{array}{ccc}
 0 & 1 & 0 \\
 1 & 0 & 1 \\
\end{array}\;,\qquad
\text{OA$_{3,2,1}^{(3)}$}=\begin{array}{ccc}
 0 & 0 & 1 \\
 1 & 1 & 0 \\
\end{array}\;,\qquad
\text{OA$_{3,2,1}^{(4)}$}=\begin{array}{ccc}
 0 & 0 & 0 \\
 1 & 1 & 1 \\
\end{array}\;,\\
\end{align*}
\vspace{-35pt}
\begin{align*}
\text{OA$_{3,2,1}^{(5)}$}=\begin{array}{ccc}
 0 & 0 & 1 \\
 0 & 1 & 0 \\
 1 & 0 & 0 \\
 1 & 1 & 1 \\
\end{array}\;,\qquad
\text{OA$_{3,2,1}^{(6)}$}=\begin{array}{ccc}
 0 & 0 & 0 \\
 0 & 1 & 1 \\
 1 & 0 & 1 \\
 1 & 1 & 0 \\
\end{array}\;.
\end{align*}
\hspace{0.5pt}

\normalsize
For strength 2, the set of orthogonal arrays $\mathscr{O\!A}(3,2,2)$ is in one-to-one correspondence with the polytope $P(3,2,2)$. The generating states of the Hilbert basis of the cone $C(3,2,2)$ are the last two elements shown in Eq.~\eqref{g31a1}, i.e. $\ket{g_{3,2,1}^{(5)}}$ and $\ket{g_{3,2,1}^{(6)}}$, so that
\be
\ket{g_{3,2,2}^{(1)}}=\ket{g_{3,2,1}^{(5)}}\;,\qquad\ket{g_{3,2,2}^{(2)}}=\ket{g_{3,2,1}^{(6)}}\;.
\ee 

\subsection{Four qubits}\label{fq}
For a 4-qubit system, the set of orthogonal arrays $\mathscr{O\!A}(4,2,1)$ is in one-to-one correspondence with the polytope $P(4,2,1)$. There are 48 generating states of the Hilbert basis of the cone $C(4,2,1)$,
\be 
\begin{gathered}
\ket{g_{4,2,1}^{(1)}}=\ket{0111}+\ket{1000}\;,\\
\ket{g_{4,2,1}^{(2)}}=\ket{0110}+\ket{1001}\;,\\
\vdots\\
\ket{g_{4,2,1}^{(48)}}=\ket{0000}+\ket{0000}+\ket{0111}+\ket{1011}+\ket{1101}+\ket{1110}
\;,\\
\end{gathered}
\ee
The generating states, which have been suppressed here for brevity, can be easily read out from the corresponding orthogonal arrays, listed below. 

\tiny
\begin{align*}
\text{OA$_{4,2,1}^{(1)}$}=
\begin{array}{cccc}
 0 & 1 & 1 & 1 \\
 1 & 0 & 0 & 0 \\
\end{array}
 \;,\qquad  
\text{OA$_{4,2,1}^{(2)}$}=
\begin{array}{cccc}
 0 & 1 & 1 & 0 \\
 1 & 0 & 0 & 1 \\
\end{array}
 \;,\qquad  
\text{OA$_{4,2,1}^{(3)}$}=
\begin{array}{cccc}
 0 & 1 & 0 & 1 \\
 1 & 0 & 1 & 0 \\
\end{array}
\;,\qquad  
\text{OA$_{4,2,1}^{(4)}$}=\begin{array}{cccc}
 0 & 1 & 0 & 0 \\
 1 & 0 & 1 & 1 \\
\end{array}
\;,\\
\end{align*}
\vspace{-35pt}
\begin{align*}  
\text{OA$_{4,2,1}^{(5)}$}=\begin{array}{cccc}
 0 & 0 & 1 & 1 \\
 1 & 1 & 0 & 0 \\
\end{array}
\;,\qquad
\text{OA$_{4,2,1}^{(6)}$}=\begin{array}{cccc}
 0 & 0 & 1 & 0 \\
 1 & 1 & 0 & 1 \\
\end{array}
 \;,\qquad 
 \text{OA$_{4,2,1}^{(7)}$}=\begin{array}{cccc}
 0 & 0 & 0 & 1 \\
 1 & 1 & 1 & 0 \\
\end{array}
 \;,\qquad   
\text{OA$_{4,2,1}^{(8)}$}=\begin{array}{cccc}
 0 & 0 & 0 & 0 \\
 1 & 1 & 1 & 1 \\
\end{array}
 \;,\\
 \end{align*}
\vspace{-35pt}
\begin{align*} 
\text{OA$_{4,2,1}^{(9)}$}=
\begin{array}{cccc}
 0 & 1 & 0 & 1 \\
 0 & 1 & 1 & 0 \\
 1 & 0 & 0 & 0 \\
 1 & 0 & 1 & 1 \\
\end{array}
 \;,\qquad  
\text{OA$_{4,2,1}^{(10)}$}=\begin{array}{cccc}
 0 & 1 & 0 & 0 \\
 0 & 1 & 1 & 1 \\
 1 & 0 & 0 & 1 \\
 1 & 0 & 1 & 0 \\
\end{array}
 \;,\qquad  
\text{OA$_{4,2,1}^{(11)}$}=\begin{array}{cccc}
 0 & 0 & 1 & 1 \\
 0 & 1 & 1 & 0 \\
 1 & 0 & 0 & 0 \\
 1 & 1 & 0 & 1 \\
\end{array}
 \;,\qquad  
\text{OA$_{4,2,1}^{(12)}$}=\begin{array}{cccc}
 0 & 0 & 1 & 1 \\
 0 & 1 & 0 & 1 \\
 1 & 0 & 0 & 0 \\
 1 & 1 & 1 & 0 \\
\end{array}
 \;,\\
 \end{align*}
\vspace{-35pt}
\begin{align*} 
 \text{OA$_{4,2,1}^{(13)}$}=\begin{array}{cccc}
 0 & 0 & 1 & 1 \\
 0 & 1 & 0 & 0 \\
 1 & 0 & 1 & 0 \\
 1 & 1 & 0 & 1 \\
\end{array}
 \;,\qquad 
\text{OA$_{4,2,1}^{(14)}$}=\begin{array}{cccc}
 0 & 0 & 1 & 1 \\
 0 & 1 & 0 & 0 \\
 1 & 0 & 0 & 1 \\
 1 & 1 & 1 & 0 \\
\end{array}
 \;,\qquad   
\text{OA$_{4,2,1}^{(15)}$}=\begin{array}{cccc}
 0 & 0 & 1 & 1 \\
 0 & 1 & 0 & 0 \\
 1 & 0 & 0 & 0 \\
 1 & 1 & 1 & 1 \\
\end{array}
 \;,\qquad 
 \text{OA$_{4,2,1}^{(16)}$}=\begin{array}{cccc}
 0 & 0 & 1 & 0 \\
 0 & 1 & 1 & 1 \\
 1 & 0 & 0 & 1 \\
 1 & 1 & 0 & 0 \\
\end{array}
 \;,\\ 
 \end{align*}
\vspace{-35pt}
\begin{align*}
\text{OA$_{4,2,1}^{(17)}$}=\begin{array}{cccc}
 0 & 0 & 1 & 0 \\
 0 & 1 & 0 & 1 \\
 1 & 0 & 1 & 1 \\
 1 & 1 & 0 & 0 \\
\end{array}
 \;,\qquad  
\text{OA$_{4,2,1}^{(18)}$}=\begin{array}{cccc}
 0 & 0 & 1 & 0 \\
 0 & 1 & 0 & 1 \\
 1 & 0 & 0 & 1 \\
 1 & 1 & 1 & 0 \\
\end{array}
 \;,\qquad 
\text{OA$_{4,2,1}^{(19)}$}=\begin{array}{cccc}
 0 & 0 & 1 & 0 \\
 0 & 1 & 0 & 1 \\
 1 & 0 & 0 & 0 \\
 1 & 1 & 1 & 1 \\
\end{array}
 \;,\qquad  
\text{OA$_{4,2,1}^{(20)}$}=\begin{array}{cccc}
 0 & 0 & 1 & 0 \\
 0 & 1 & 0 & 0 \\
 1 & 0 & 0 & 1 \\
 1 & 1 & 1 & 1 \\
\end{array}
 \;,\\ 
 \end{align*}
\vspace{-35pt}
\begin{align*}
\text{OA$_{4,2,1}^{(21)}$}=\begin{array}{cccc}
 0 & 0 & 0 & 1 \\
 0 & 1 & 1 & 1 \\
 1 & 0 & 1 & 0 \\
 1 & 1 & 0 & 0 \\
\end{array}
 \;,\qquad  
\text{OA$_{4,2,1}^{(22)}$}=\begin{array}{cccc}
 0 & 0 & 0 & 1 \\
 0 & 1 & 1 & 0 \\
 1 & 0 & 1 & 1 \\
 1 & 1 & 0 & 0 \\
\end{array}
 \;,\qquad 
\text{OA$_{4,2,1}^{(23)}$}=\begin{array}{cccc}
 0 & 0 & 0 & 1 \\
 0 & 1 & 1 & 0 \\
 1 & 0 & 1 & 0 \\
 1 & 1 & 0 & 1 \\
\end{array}
 \;,\qquad  
\text{OA$_{4,2,1}^{(24)}$}=\begin{array}{cccc}
 0 & 0 & 0 & 1 \\
 0 & 1 & 1 & 0 \\
 1 & 0 & 0 & 0 \\
 1 & 1 & 1 & 1 \\
\end{array}
 \;,\\
 \end{align*}
\vspace{-35pt}
\begin{align*}
 \text{OA$_{4,2,1}^{(25)}$}=\begin{array}{cccc}
 0 & 0 & 0 & 1 \\
 0 & 1 & 0 & 0 \\
 1 & 0 & 1 & 0 \\
 1 & 1 & 1 & 1 \\
\end{array}
 \;,\qquad 
 \text{OA$_{4,2,1}^{(26)}$}=\begin{array}{cccc}
 0 & 0 & 0 & 1 \\
 0 & 0 & 1 & 0 \\
 1 & 1 & 0 & 0 \\
 1 & 1 & 1 & 1 \\
\end{array}
 \;,\qquad  
 \text{OA$_{4,2,1}^{(27)}$}=\begin{array}{cccc}
 0 & 0 & 0 & 0 \\
 0 & 1 & 1 & 1 \\
 1 & 0 & 1 & 1 \\
 1 & 1 & 0 & 0 \\
\end{array}
 \;,\qquad 
 \text{OA$_{4,2,1}^{(28)}$}=\begin{array}{cccc}
 0 & 0 & 0 & 0 \\
 0 & 1 & 1 & 1 \\
 1 & 0 & 1 & 0 \\
 1 & 1 & 0 & 1 \\
\end{array}
 \;,\\
 \end{align*}
\vspace{-35pt}
\begin{align*}   
\text{OA$_{4,2,1}^{(29)}$}=\begin{array}{cccc}
 0 & 0 & 0 & 0 \\
 0 & 1 & 1 & 1 \\
 1 & 0 & 0 & 1 \\
 1 & 1 & 1 & 0 \\
\end{array}
 \;,\qquad 
\text{OA$_{4,2,1}^{(30)}$}=\begin{array}{cccc}
 0 & 0 & 0 & 0 \\
 0 & 1 & 1 & 0 \\
 1 & 0 & 1 & 1 \\
 1 & 1 & 0 & 1 \\
\end{array}
 \;,\qquad  
\text{OA$_{4,2,1}^{(31)}$}=\begin{array}{cccc}
 0 & 0 & 0 & 0 \\
 0 & 1 & 0 & 1 \\
 1 & 0 & 1 & 1 \\
 1 & 1 & 1 & 0 \\
\end{array}
 \;,\qquad  
\text{OA$_{4,2,1}^{(32)}$}=\begin{array}{cccc}
 0 & 0 & 0 & 0 \\
 0 & 0 & 1 & 1 \\
 1 & 1 & 0 & 1 \\
 1 & 1 & 1 & 0 \\
\end{array}
 \;,\\
 \end{align*}
\vspace{-35pt}
\begin{align*}
\text{OA$_{4,2,1}^{(33)}$}=\begin{array}{cccc}
 0 & 0 & 1 & 1 \\
 0 & 1 & 0 & 1 \\
 0 & 1 & 1 & 0 \\
 1 & 0 & 0 & 0 \\
 1 & 0 & 0 & 0 \\
 1 & 1 & 1 & 1 \\
\end{array}
 \;,\qquad  
\text{OA$_{4,2,1}^{(34)}$}=\begin{array}{cccc}
 0 & 0 & 1 & 1 \\
 0 & 1 & 0 & 0 \\
 0 & 1 & 0 & 0 \\
 1 & 0 & 0 & 1 \\
 1 & 0 & 1 & 0 \\
 1 & 1 & 1 & 1 \\
\end{array}
 \;,\qquad   
\text{OA$_{4,2,1}^{(35)}$}=\begin{array}{cccc}
 0 & 0 & 1 & 1 \\
 0 & 0 & 1 & 1 \\
 0 & 1 & 0 & 0 \\
 1 & 0 & 0 & 0 \\
 1 & 1 & 0 & 1 \\
 1 & 1 & 1 & 0 \\
\end{array}
 \;,\qquad  
\text{OA$_{4,2,1}^{(36)}$}=\begin{array}{cccc}
 0 & 0 & 1 & 0 \\
 0 & 1 & 0 & 1 \\
 0 & 1 & 0 & 1 \\
 1 & 0 & 0 & 0 \\
 1 & 0 & 1 & 1 \\
 1 & 1 & 1 & 0 \\
\end{array}
 \;,\\ 
 \end{align*}
\vspace{-35pt}
\begin{align*}  
\text{OA$_{4,2,1}^{(37)}$}=\begin{array}{cccc}
 0 & 0 & 1 & 0 \\
 0 & 1 & 0 & 0 \\
 0 & 1 & 1 & 1 \\
 1 & 0 & 0 & 1 \\
 1 & 0 & 0 & 1 \\
 1 & 1 & 1 & 0 \\
\end{array}
 \;,\qquad  
\text{OA$_{4,2,1}^{(38)}$}=\begin{array}{cccc}
 0 & 0 & 1 & 0 \\
 0 & 0 & 1 & 0 \\
 0 & 1 & 0 & 1 \\
 1 & 0 & 0 & 1 \\
 1 & 1 & 0 & 0 \\
 1 & 1 & 1 & 1 \\
\end{array}
 \;,\qquad     
\text{OA$_{4,2,1}^{(39)}$}=\begin{array}{cccc}
 0 & 0 & 0 & 1 \\
 0 & 1 & 1 & 0 \\
 0 & 1 & 1 & 0 \\
 1 & 0 & 0 & 0 \\
 1 & 0 & 1 & 1 \\
 1 & 1 & 0 & 1 \\
\end{array}
 \;,\qquad  
\text{OA$_{4,2,1}^{(40)}$}=\begin{array}{cccc}
 0 & 0 & 0 & 1 \\
 0 & 1 & 0 & 0 \\
 0 & 1 & 1 & 1 \\
 1 & 0 & 1 & 0 \\
 1 & 0 & 1 & 0 \\
 1 & 1 & 0 & 1 \\
\end{array}
 \;,\\ 
 \end{align*}
\vspace{-35pt}
\begin{align*}   
\text{OA$_{4,2,1}^{(41)}$}=\begin{array}{cccc}
 0 & 0 & 0 & 1 \\
 0 & 0 & 1 & 0 \\
 0 & 1 & 1 & 1 \\
 1 & 0 & 1 & 1 \\
 1 & 1 & 0 & 0 \\
 1 & 1 & 0 & 0 \\
\end{array}
 \;,\qquad  
\text{OA$_{4,2,1}^{(42)}$}=\begin{array}{cccc}
 0 & 0 & 0 & 1 \\
 0 & 0 & 1 & 0 \\
 0 & 1 & 0 & 0 \\
 1 & 0 & 0 & 0 \\
 1 & 1 & 1 & 1 \\
 1 & 1 & 1 & 1 \\
\end{array}
 \;,\qquad 
\text{OA$_{4,2,1}^{(43)}$}=\begin{array}{cccc}
 0 & 0 & 0 & 1 \\
 0 & 0 & 0 & 1 \\
 0 & 1 & 1 & 0 \\
 1 & 0 & 1 & 0 \\
 1 & 1 & 0 & 0 \\
 1 & 1 & 1 & 1 \\
\end{array}
 \;,\qquad     
\text{OA$_{4,2,1}^{(44)}$}=\begin{array}{cccc}
 0 & 0 & 0 & 0 \\
 0 & 1 & 1 & 1 \\
 0 & 1 & 1 & 1 \\
 1 & 0 & 0 & 1 \\
 1 & 0 & 1 & 0 \\
 1 & 1 & 0 & 0 \\
\end{array}
 \;,\\ 
 \end{align*}
\vspace{-35pt}
\begin{align*}  
\text{OA$_{4,2,1}^{(45)}$}=\begin{array}{cccc}
 0 & 0 & 0 & 0 \\
 0 & 1 & 0 & 1 \\
 0 & 1 & 1 & 0 \\
 1 & 0 & 1 & 1 \\
 1 & 0 & 1 & 1 \\
 1 & 1 & 0 & 0 \\
\end{array}
 \;,\qquad 
\text{OA$_{4,2,1}^{(46)}$}=\begin{array}{cccc}
 0 & 0 & 0 & 0 \\
 0 & 0 & 1 & 1 \\
 0 & 1 & 1 & 0 \\
 1 & 0 & 1 & 0 \\
 1 & 1 & 0 & 1 \\
 1 & 1 & 0 & 1 \\
\end{array}
\;,\qquad   
\text{OA$_{4,2,1}^{(47)}$}=\begin{array}{cccc}
0 & 0 & 0 & 0 \\
0 & 0 & 1 & 1 \\
0 & 1 & 0 & 1 \\
1 & 0 & 0 & 1 \\
1 & 1 & 1 & 0 \\
1 & 1 & 1 & 0 \\
\end{array}
\;,\qquad  
\text{OA$_{4,2,1}^{(48)}$}=\begin{array}{cccc}
0 & 0 & 0 & 0 \\
0 & 0 & 0 & 0 \\
0 & 1 & 1 & 1 \\
1 & 0 & 1 & 1 \\
1 & 1 & 0 & 1 \\
1 & 1 & 1 & 0 \\
\end{array}\;.
\end{align*}
\hspace{5pt}\\
\normalsize
For strength 2, the set of orthogonal arrays $\mathscr{O\!A}(4,2,2)$ is in one-to-one correspondence with the polytope $P(4,2,2)$. The generating states of the Hilbert basis of the cone $C(4,2,2)$ are 
\be
\begin{gathered}
\ket{g_{4,2,2}^{(1)}}=\ket{0010}+\ket{0011}+\ket{0100}+\ket{0101}+\ket{1000}+\ket{1001}
+\ket{1110}+\ket{1111}\;,\\
\ket{g_{4,2,2}^{(2)}}=\ket{0001}+\ket{0011}+\ket{0100}+\ket{0110}+\ket{1000}+\ket{1010}+\ket{1101}+\ket{1111}\;,\\
\vdots\\
\ket{g_{4,2,2}^{(26)}}=\ket{0000}+\ket{0000}+\ket{0011}+\ket{0101}+\ket{0110}+\ket{0111}+\ket{1001}+\ket{1010}\\ +\ket{1011}+\ket{1100}+\ket{1101}+\ket{1110}\;.\\
\end{gathered}
\ee

The corresponding orthogonal arrays are listed below.

\tiny
\begin{align*}
\text{OA$_{4,2,2}^{(1)}$}=
\begin{array}{cccc}
 0 & 0 & 1 & 0 \\
 0 & 0 & 1 & 1 \\
 0 & 1 & 0 & 0 \\
 0 & 1 & 0 & 1 \\
 1 & 0 & 0 & 0 \\
 1 & 0 & 0 & 1 \\
 1 & 1 & 1 & 0 \\
 1 & 1 & 1 & 1 \\
\end{array}
\;,\qquad
\text{OA$_{4,2,2}^{(2)}$}=\begin{array}{cccc}
 0 & 0 & 0 & 1 \\
 0 & 0 & 1 & 1 \\
 0 & 1 & 0 & 0 \\
 0 & 1 & 1 & 0 \\
 1 & 0 & 0 & 0 \\
 1 & 0 & 1 & 0 \\
 1 & 1 & 0 & 1 \\
 1 & 1 & 1 & 1 \\
\end{array}
\;,\qquad
\text{OA$_{4,2,2}^{(3)}$}=\begin{array}{cccc}
 0 & 0 & 0 & 1 \\
 0 & 0 & 1 & 0 \\
 0 & 1 & 0 & 1 \\
 0 & 1 & 1 & 0 \\
 1 & 0 & 0 & 0 \\
 1 & 0 & 1 & 1 \\
 1 & 1 & 0 & 0 \\
 1 & 1 & 1 & 1 \\
\end{array}
\;,\qquad
\text{OA$_{4,2,2}^{(4)}$}=\begin{array}{cccc}
 0 & 0 & 0 & 1 \\
 0 & 0 & 1 & 0 \\
 0 & 1 & 0 & 0 \\
 0 & 1 & 1 & 1 \\
 1 & 0 & 0 & 1 \\
 1 & 0 & 1 & 0 \\
 1 & 1 & 0 & 0 \\
 1 & 1 & 1 & 1 \\
\end{array}
\;,\qquad
\text{OA$_{4,2,2}^{(5)}$}=\begin{array}{cccc}
 0 & 0 & 0 & 1 \\
 0 & 0 & 1 & 0 \\
 0 & 1 & 0 & 0 \\
 0 & 1 & 1 & 1 \\
 1 & 0 & 0 & 0 \\
 1 & 0 & 1 & 1 \\
 1 & 1 & 0 & 1 \\
 1 & 1 & 1 & 0 \\
\end{array}
\;,\\
\end{align*}
\vspace{-35pt}
\begin{align*}
\text{OA$_{4,2,2}^{(6)}$}=\begin{array}{cccc}
 0 & 0 & 0 & 0 \\
 0 & 0 & 1 & 1 \\
 0 & 1 & 0 & 1 \\
 0 & 1 & 1 & 0 \\
 1 & 0 & 0 & 1 \\
 1 & 0 & 1 & 0 \\
 1 & 1 & 0 & 0 \\
 1 & 1 & 1 & 1 \\
\end{array}
\;,\qquad
\text{OA$_{4,2,2}^{(7)}$}=\begin{array}{cccc}
 0 & 0 & 0 & 0 \\
 0 & 0 & 1 & 1 \\
 0 & 1 & 0 & 1 \\
 0 & 1 & 1 & 0 \\
 1 & 0 & 0 & 0 \\
 1 & 0 & 1 & 1 \\
 1 & 1 & 0 & 1 \\
 1 & 1 & 1 & 0 \\
\end{array}
\;,\qquad
\text{OA$_{4,2,2}^{(8)}$}=\begin{array}{cccc}
 0 & 0 & 0 & 0 \\
 0 & 0 & 1 & 1 \\
 0 & 1 & 0 & 0 \\
 0 & 1 & 1 & 1 \\
 1 & 0 & 0 & 1 \\
 1 & 0 & 1 & 0 \\
 1 & 1 & 0 & 1 \\
 1 & 1 & 1 & 0 \\
\end{array}
\;,\qquad
\text{OA$_{4,2,2}^{(9)}$}=\begin{array}{cccc}
 0 & 0 & 0 & 0 \\
 0 & 0 & 1 & 0 \\
 0 & 1 & 0 & 1 \\
 0 & 1 & 1 & 1 \\
 1 & 0 & 0 & 1 \\
 1 & 0 & 1 & 1 \\
 1 & 1 & 0 & 0 \\
 1 & 1 & 1 & 0 \\
\end{array}
\;,\qquad
\text{OA$_{4,2,2}^{(10)}$}=\begin{array}{cccc}
 0 & 0 & 0 & 0 \\
 0 & 0 & 0 & 1 \\
 0 & 1 & 1 & 0 \\
 0 & 1 & 1 & 1 \\
 1 & 0 & 1 & 0 \\
 1 & 0 & 1 & 1 \\
 1 & 1 & 0 & 0 \\
 1 & 1 & 0 & 1 \\
\end{array}
\;,\\
\end{align*}
\vspace{-35pt}
\begin{align*}
\text{OA$_{4,2,2}^{(11)}$}=\begin{array}{cccc}
 0 & 0 & 0 & 1 \\
 0 & 0 & 1 & 0 \\
 0 & 0 & 1 & 1 \\
 0 & 1 & 0 & 0 \\
 0 & 1 & 0 & 1 \\
 0 & 1 & 1 & 0 \\
 1 & 0 & 0 & 0 \\
 1 & 0 & 0 & 1 \\
 1 & 0 & 1 & 0 \\
 1 & 1 & 0 & 0 \\
 1 & 1 & 1 & 1 \\
 1 & 1 & 1 & 1 \\
\end{array}
\;,\qquad
\text{OA$_{4,2,2}^{(12)}$}=\begin{array}{cccc} 
 0 & 0 & 0 & 1 \\
 0 & 0 & 1 & 0 \\
 0 & 0 & 1 & 1 \\
 0 & 1 & 0 & 0 \\
 0 & 1 & 0 & 1 \\
 0 & 1 & 1 & 0 \\
 1 & 0 & 0 & 0 \\
 1 & 0 & 0 & 0 \\
 1 & 0 & 1 & 1 \\
 1 & 1 & 0 & 1 \\
 1 & 1 & 1 & 0 \\
 1 & 1 & 1 & 1 \\
\end{array}
\;,\qquad
\text{OA$_{4,2,2}^{(13)}$}=\begin{array}{cccc} 
 0 & 0 & 0 & 1 \\
 0 & 0 & 1 & 0 \\
 0 & 0 & 1 & 1 \\
 0 & 1 & 0 & 0 \\
 0 & 1 & 0 & 0 \\
 0 & 1 & 1 & 1 \\
 1 & 0 & 0 & 0 \\
 1 & 0 & 0 & 1 \\
 1 & 0 & 1 & 0 \\
 1 & 1 & 0 & 1 \\
 1 & 1 & 1 & 0 \\
 1 & 1 & 1 & 1 \\
\end{array}
\;,\qquad
\text{OA$_{4,2,2}^{(14)}$}=\begin{array}{cccc} 
 0 & 0 & 0 & 1 \\
 0 & 0 & 1 & 0 \\
 0 & 0 & 1 & 0 \\
 0 & 1 & 0 & 0 \\
 0 & 1 & 0 & 1 \\
 0 & 1 & 1 & 1 \\
 1 & 0 & 0 & 0 \\
 1 & 0 & 0 & 1 \\
 1 & 0 & 1 & 1 \\
 1 & 1 & 0 & 0 \\
 1 & 1 & 1 & 0 \\
 1 & 1 & 1 & 1 \\
\end{array}
\;,\\
\end{align*}
\vspace{-35pt}
\begin{align*}
\text{OA$_{4,2,2}^{(15)}$}=\begin{array}{cccc} 
 0 & 0 & 0 & 1 \\
 0 & 0 & 0 & 1 \\
 0 & 0 & 1 & 0 \\
 0 & 1 & 0 & 0 \\
 0 & 1 & 1 & 0 \\
 0 & 1 & 1 & 1 \\
 1 & 0 & 0 & 0 \\
 1 & 0 & 1 & 0 \\
 1 & 0 & 1 & 1 \\
 1 & 1 & 0 & 0 \\
 1 & 1 & 0 & 1 \\
 1 & 1 & 1 & 1 \\
\end{array}
\;,\qquad
\text{OA$_{4,2,2}^{(16)}$}=\begin{array}{cccc} 
 0 & 0 & 0 & 0 \\
 0 & 0 & 1 & 1 \\
 0 & 0 & 1 & 1 \\
 0 & 1 & 0 & 0 \\
 0 & 1 & 0 & 1 \\
 0 & 1 & 1 & 0 \\
 1 & 0 & 0 & 0 \\
 1 & 0 & 0 & 1 \\
 1 & 0 & 1 & 0 \\
 1 & 1 & 0 & 1 \\
 1 & 1 & 1 & 0 \\
 1 & 1 & 1 & 1 \\
\end{array}
\;,\qquad
\text{OA$_{4,2,2}^{(17)}$}=\begin{array}{cccc} 
 0 & 0 & 0 & 0 \\
 0 & 0 & 1 & 0 \\
 0 & 0 & 1 & 1 \\
 0 & 1 & 0 & 1 \\
 0 & 1 & 0 & 1 \\
 0 & 1 & 1 & 0 \\
 1 & 0 & 0 & 0 \\
 1 & 0 & 0 & 1 \\
 1 & 0 & 1 & 1 \\
 1 & 1 & 0 & 0 \\
 1 & 1 & 1 & 0 \\
 1 & 1 & 1 & 1 \\
\end{array}
\;,\qquad
\text{OA$_{4,2,2}^{(18)}$}=\begin{array}{cccc} 
 0 & 0 & 0 & 0 \\
 0 & 0 & 1 & 0 \\
 0 & 0 & 1 & 1 \\
 0 & 1 & 0 & 0 \\
 0 & 1 & 0 & 1 \\
 0 & 1 & 1 & 1 \\
 1 & 0 & 0 & 1 \\
 1 & 0 & 0 & 1 \\
 1 & 0 & 1 & 0 \\
 1 & 1 & 0 & 0 \\
 1 & 1 & 1 & 0 \\
 1 & 1 & 1 & 1 \\
\end{array}
\;,\\
\end{align*}
\vspace{-35pt}
\begin{align*}
\text{OA$_{4,2,2}^{(19)}$}=\begin{array}{cccc} 
 0 & 0 & 0 & 0 \\
 0 & 0 & 1 & 0 \\
 0 & 0 & 1 & 1 \\
 0 & 1 & 0 & 0 \\
 0 & 1 & 0 & 1 \\
 0 & 1 & 1 & 1 \\
 1 & 0 & 0 & 0 \\
 1 & 0 & 0 & 1 \\
 1 & 0 & 1 & 1 \\
 1 & 1 & 0 & 1 \\
 1 & 1 & 1 & 0 \\
 1 & 1 & 1 & 0 \\
\end{array}
\;,\qquad
\text{OA$_{4,2,2}^{(20)}$}=\begin{array}{cccc} 
 0 & 0 & 0 & 0 \\
 0 & 0 & 0 & 1 \\
 0 & 0 & 1 & 1 \\
 0 & 1 & 0 & 1 \\
 0 & 1 & 1 & 0 \\
 0 & 1 & 1 & 0 \\
 1 & 0 & 0 & 0 \\
 1 & 0 & 1 & 0 \\
 1 & 0 & 1 & 1 \\
 1 & 1 & 0 & 0 \\
 1 & 1 & 0 & 1 \\
 1 & 1 & 1 & 1 \\
\end{array}
\;,\qquad
\text{OA$_{4,2,2}^{(21)}$}=\begin{array}{cccc} 
 0 & 0 & 0 & 0 \\
 0 & 0 & 0 & 1 \\
 0 & 0 & 1 & 1 \\
 0 & 1 & 0 & 0 \\
 0 & 1 & 1 & 0 \\
 0 & 1 & 1 & 1 \\
 1 & 0 & 0 & 1 \\
 1 & 0 & 1 & 0 \\
 1 & 0 & 1 & 0 \\
 1 & 1 & 0 & 0 \\
 1 & 1 & 0 & 1 \\
 1 & 1 & 1 & 1 \\
\end{array}
\;,\qquad
\text{OA$_{4,2,2}^{(22)}$}=\begin{array}{cccc} 
 0 & 0 & 0 & 0 \\
 0 & 0 & 0 & 1 \\
 0 & 0 & 1 & 1 \\
 0 & 1 & 0 & 0 \\
 0 & 1 & 1 & 0 \\
 0 & 1 & 1 & 1 \\
 1 & 0 & 0 & 0 \\
 1 & 0 & 1 & 0 \\
 1 & 0 & 1 & 1 \\
 1 & 1 & 0 & 1 \\
 1 & 1 & 0 & 1 \\
 1 & 1 & 1 & 0 \\
\end{array}
\;,\\
\end{align*}
\vspace{-35pt}
\begin{align*}
\text{OA$_{4,2,2}^{(23)}$}=\begin{array}{cccc} 
 0 & 0 & 0 & 0 \\
 0 & 0 & 0 & 1 \\
 0 & 0 & 1 & 0 \\
 0 & 1 & 0 & 1 \\
 0 & 1 & 1 & 0 \\
 0 & 1 & 1 & 1 \\
 1 & 0 & 0 & 1 \\
 1 & 0 & 1 & 0 \\
 1 & 0 & 1 & 1 \\
 1 & 1 & 0 & 0 \\
 1 & 1 & 0 & 0 \\
 1 & 1 & 1 & 1 \\
\end{array}
\;,\qquad
\text{OA$_{4,2,2}^{(24)}$}=\begin{array}{cccc} 
 0 & 0 & 0 & 0 \\
 0 & 0 & 0 & 1 \\
 0 & 0 & 1 & 0 \\
 0 & 1 & 0 & 1 \\
 0 & 1 & 1 & 0 \\
 0 & 1 & 1 & 1 \\
 1 & 0 & 0 & 0 \\
 1 & 0 & 1 & 1 \\
 1 & 0 & 1 & 1 \\
 1 & 1 & 0 & 0 \\
 1 & 1 & 0 & 1 \\
 1 & 1 & 1 & 0 \\
\end{array}
\;,\qquad
\text{OA$_{4,2,2}^{(25)}$}=\begin{array}{cccc} 
 0 & 0 & 0 & 0 \\
 0 & 0 & 0 & 1 \\
 0 & 0 & 1 & 0 \\
 0 & 1 & 0 & 0 \\
 0 & 1 & 1 & 1 \\
 0 & 1 & 1 & 1 \\
 1 & 0 & 0 & 1 \\
 1 & 0 & 1 & 0 \\
 1 & 0 & 1 & 1 \\
 1 & 1 & 0 & 0 \\
 1 & 1 & 0 & 1 \\
 1 & 1 & 1 & 0 \\
\end{array}
\;,\qquad
\text{OA$_{4,2,2}^{(26)}$}=\begin{array}{cccc}
 0 & 0 & 0 & 0 \\
 0 & 0 & 0 & 0 \\
 0 & 0 & 1 & 1 \\
 0 & 1 & 0 & 1 \\
 0 & 1 & 1 & 0 \\
 0 & 1 & 1 & 1 \\
 1 & 0 & 0 & 1 \\
 1 & 0 & 1 & 0 \\
 1 & 0 & 1 & 1 \\
 1 & 1 & 0 & 0 \\
 1 & 1 & 0 & 1 \\
 1 & 1 & 1 & 0 \\
\end{array}\;.
\end{align*}

\normalsize
For strength 3, the set of orthogonal arrays $\mathscr{O\!A}(4,2,3)$ is in one-to-one correspondence with the polytope $P(4,2,3)$. The generating states of the Hilbert basis of the cone $ C(4,2,3)$ are $\ket{g_{4,2,2}^{(5)}}$ and $\ket{g_{4,2,2}^{(6)}}$ given explicitly above, i.e.
\be
\ket{g_{4,2,3}^{(1)}}=\ket{g_{4,2,2}^{(5)}} \;,\qquad \ket{g_{4,2,3}^{(2)}}= \ket{g_{4,2,2}^{(6)}}\;. 
\ee  

\section{Polynomial invariants for 4-qubit systems}\label{App2}
In this appendix, we compute a complete family of 4-qubit invariants for the generating states belonging to the sets $\mathcal G(4,2,k)$, with $k=1,\,2,\,3$. There is an infinite number of entanglement classes for four qubits in the SLOCC classification, which can be organized into nine continuous families \cite{VDMV02}. Of the nine families, only one family is generic. States belonging to the generic family can be distinguished because their hyperdeterminant $\Delta$ is nonvanishing \cite{HLT17}. The other eight families, having vanishing hyperdeterminant, belong to a set of measure zero in the 4-qubit Hilbert space. It is possible to discriminate among them by considering the following set of entanglement invariants (see Tables IV to VI of Ref.~\cite{HLT17}): the Cayley invariant $H$ of degree two, the invariants $L$ and $M$ of degree four and the invariant $D_{xy}$ of degree six \cite{LT03}. The degree three invariant $T$ is known as the \emph{catalecticant} \cite{P99}, while the invariants $P$, $S_1$, $S_2$ and $S_3$ are defined in Eq.~$(31)$ of Ref.~\cite{HLT17}). 

{\fontsize{6}{8}\selectfont
\begin{center}
\begin{tabular}{c|c}
\{$N=4$,$k=1$,$i=1-48$\}&\{$\Delta$, $H$, $T$, $L$, $P$, $M$, $D_{xy}$, $S_1$, $S_2$, $S_3$\}\\ \hline
\{4,1,1\} &\{0,25/192,-253/13824,0,-1/8,0,0,1/4,1/4,1/4\}\\
\{4,1,2\} &\{0,1/192,-1/13824,0,0,0,0,1/4,1/4,1/4\}\\
\{4,1,3\} &\{0,1/192,-1/13824,0,0,0,0,1/4,1/4,1/4\}\\
\{4,1,4\} &\{0,1/192,-1/13824,0,0,0,0,1/4,1/4,1/4\}\\
\{4,1,5\} &\{0,1/192,-1/13824,0,0,0,0,1/4,1/4,1/4\}\\
\{4,1,6\} &\{0,1/192,-1/13824,0,0,0,0,1/4,1/4,1/4\}\\
\{4,1,7\} &\{0,1/192,-1/13824,0,0,0,0,1/4,1/4,1/4\}\\
\{4,1,8\} &\{0,1/192,-1/13824,0,0,0,0,1/4,1/4,1/4\}\\
\{4,1,9\} &\{0,0,0,0,0,0,0,0,0,0\}\\
\{4,1,10\} &\{0,1/192,-1/13824,0,0,1/16,0,1/4,0,-1/4\}\\
\{4,1,11\} &\{0,1/192,-1/8704,-1/16,0,0,0,0,1/4,0\}\\
\{4,1,12\} &\{0,1/192,-7/235008,1/16,0,-1/16,0,-1/4,-1/4,1/4\}\\
\{4,1,13\} &\{0,1/192,-7/235008,1/16,0,-1/16,0,-1/4,-1/4,1/4\}\\
\{4,1,14\} &\{0,1/192,-1/8704,-1/16,0,0,0,0,1/4,0\}\\
\{4,1,15\} &\{0,1/192,-1/13824,0,0,1/16,0,1/4,0,-1/4\}\\
\{4,1,16\} &\{0,1/192,-1/8704,-1/16,0,0,0,0,1/4,0\}\\
\{4,1,17\} &\{0,1/192,1/8704,1/16,0,0,0,0,-1/4,0\}\\
\{4,1,18\} &\{0,1/192,-1/13824,0,0,1/16,0,1/4,0,-1/4\}\\
\{4,1,19\} &\{0,1/192,-1/8704,-1/16,0,0,0,0,1/4,0\}\\
\{4,1,20\} &\{0,1/192,-7/235008,1/16,0,-1/16,0,-1/4,-1/4,1/4\}\\
\{4,1,21\} &\{0,1/192,-7/235008,1/16,0,-1/16,0,-1/4,-1/4,1/4\}\\
\{4,1,22\} &\{0,1/192,-1/8704,-1/16,0,0,0,0,1/4,0\}\\
\{4,1,23\} &\{0,1/192,-1/13824,0,0,1/16,0,1/4,0,-1/4\}\\
\{4,1,24\} &\{0,1/192,-7/235008,1/16,0,-1/16,0,-1/4,-1/4,1/4\}\\
\{4,1,25\} &\{0,1/192,-1/8704,-1/16,0,0,0,0,1/4,0\}\\
\{4,1,26\} &\{0,1/192,-1/13824,0,0,1/16,0,1/4,0,-1/4\}\\
\{4,1,27\} &\{0,0,0,0,0,0,0,0,0,0\}\\
\{4,1,28\} &\{0,1/192,-1/8704,-1/16,0,0,0,0,1/4,0\}\\
\{4,1,29\} &\{0,1/192,-7/235008,1/16,0,-1/16,0,-1/4,-1/4,1/4\}\\
\{4,1,30\} &\{0,1/192,1/8704,1/16,0,0,0,0,-1/4,0\}\\
\{4,1,31\} &\{0,1/192,-1/8704,-1/16,0,0,0,0,1/4,0\}\\
\{4,1,32\} &\{0,1/192,-1/13824,0,0,1/16,0,1/4,0,-1/4\}\\
\{4,1,33\} &\{-27/268435456,0,0,0,0,0,1/8,0,0,0\}\\
\{4,1,34\} &\{-27/268435456,0,-1/64,0,1/8,0,1/8,0,0,0\}\\
\{4,1,35\} &\{-27/268435456,0,-1/64,0,1/8,0,-1/8,0,0,0\}\\
\{4,1,36\} &\{-27/268435456,0,-1/64,0,-1/8,0,-1/8,0,0,0\}\\
\{4,1,37\} &\{-27/268435456,0,-1/64,0,-1/8,0,-1/8,0,0,0\}\\
\{4,1,38\} &\{-27/268435456,1/768,-1727/110592,0,-1/8,-1/32,1/8,-1/8,0,1/8\}\\
\{4,1,39\} &\{-27/268435456,0,-1/64,0,1/8,0,-1/8,0,0,0\}\\
\{4,1,40\} &\{-27/268435456,0,-1/64,0,-1/8,0,-1/8,0,0,0\}\\
\{4,1,41\} &\{-27/268435456,0,-1/64,0,-1/8,0,-1/8,0,0,0\}\\
\{4,1,42\} &\{-27/268435456,0,-1/64,0,-1/8,0,-1/8,0,0,0\}\\
\{4,1,43\} &\{-27/268435456,0,-1/64,0,-1/8,0,1/8,0,0,0\}\\
\{4,1,44\} &\{-27/268435456,1/768,-1729/110592,0,1/8,1/32,1/8,1/8,0,-1/8\}\\
\{4,1,45\} &\{-27/268435456,0,-1/64,0,1/8,0,1/8,0,0,0\}\\
\{4,1,46\} &\{-27/268435456,0,-1/64,0,1/8,0,1/8,0,0,0\}\\
\{4,1,47\} &\{-27/268435456,0,-1/64,0,1/8,0,1/8,0,0,0\}\\
\{4,1,48\} &\{-27/268435456,0,-1/64,0,1/8,0,-1/8,0,0,0\}
\end{tabular}
\end{center}

\begin{center}
\begin{tabular}{c|c}
\{$N=4$,$k=2$,$i=1-26$\}&\{$\Delta$, $H$, $T$, $L$, $P$, $M$, $D_{xy}$, $S_1$, $S_2$, $S_3$\}\\ \hline
\{4,2,1\} & \{0,0,-1/64,0,-1/8,0,0,0,0,0\}\\
\{4,2,2\} & \{0,0,0,0,0,0,0,0,0,0\}\\
\{4,2,3\} & \{0,0,0,0,0,0,0,0,0,0\}\\
\{4,2,4\} & \{0,0,0,0,0,0,0,0,0,0\}\\
\{4,2,5\} & \{0,1/192,-1/13824,0,0,0,0,1/4,1/4,1/4\}\\
\{4,2,6\} & \{0,1/192,-1/13824,0,0,0,0,1/4,1/4,1/4\}\\
\{4,2,7\} & \{0,0,0,0,0,0,0,0,0,0\}\\
\{4,2,8\} & \{0,0,0,0,0,0,0,0,0,0\}\\
\{4,2,9\} & \{0,0,0,0,0,0,0,0,0,0\}\\
\{4,2,10\} & \{0,0,0,0,0,0,0,0,0,0\}\\
\{4,2,11\} & \{-19683/7086739046912,27/153664,-27/60236288,0,0,0,-18/343,9/196,9/196,9/196\}\\
\{4,2,12\} & \{-19683/7086739046912,75/3136,-192347/60236288,0,-153/2744,-3/196,18/343,-3/196,9/196,3/28\}\\
\{4,2,13\} & \{-19683/7086739046912,-3429/153664,-160731/60236288,0,18/343,0,18/343,9/196,9/196,9/196\}\\
\{4,2,14\} & \{-19683/7086739046912,-3429/153664,-160731/60236288,0,18/343,0,18/343,9/196,9/196,9/196\}\\
\{4,2,15\} & \{-19683/7086739046912,-3429/153664,-160731/60236288,0,18/343,0,18/343,9/196,9/196,9/196\}\\
\{4,2,16\} & \{-19683/7086739046912,3483/153664,-171099/60236288,0,18/343,0,-18/343,9/196,9/196,9/196\}\\
\{4,2,17\} & \{-19683/7086739046912,-3429/153664,-160731/60236288,0,-18/343,0,-18/343,9/196,9/196,9/196\}\\
\{4,2,18\} & \{-19683/7086739046912,-3429/153664,-160731/60236288,0,-18/343,0,-18/343,9/196,9/196,9/196\}\\
\{4,2,19\} & \{-19683/7086739046912,3483/153664,-171099/60236288,0,-18/343,0,18/343,9/196,9/196,9/196\}\\
\{4,2,20\} & \{-19683/7086739046912,3483/153664,-171099/60236288,0,18/343,0,-18/343,9/196,9/196,9/196\}\\
\{4,2,21\} & \{-19683/7086739046912,-3429/153664,-160731/60236288,0,-18/343,0,-18/343,9/196,9/196,9/196\}\\
\{4,2,22\} & \{-19683/7086739046912,3387/153664,-150571/60236288,0,-135/2744,3/196,18/343,3/28,9/196,-3/196\}\\
\{4,2,23\} & \{-19683/7086739046912,3483/153664,-171099/60236288,0,18/343,0,-18/343,9/196,9/196,9/196\}\\
\{4,2,24\} & \{-19683/7086739046912,3483/153664,-171099/60236288,0,-18/343,0,18/343,9/196,9/196,9/196\}\\
\{4,2,25\} & \{-19683/7086739046912,-3429/153664,-160731/60236288,0,18/343,0,18/343,9/196,9/196,9/196\}\\
\{4,2,26\} & \{-19683/7086739046912,3483/153664,-171099/60236288,0,18/343,0,-18/343,9/196,9/196,9/196\}
\end{tabular}
\end{center}
\vspace{1cm}

\begin{center}
\begin{tabular}{c|c}
\{$N=4$,$k=3$,$i=1-2$\}&\{$\Delta$, $H$, $T$, $L$, $P$, $M$, $D_{xy}$, $S_1$, $S_2$, $S_3$\}\\ \hline
\{4,3,1\} & \{0,3799/65856,-6374737/1626379776,0,-18/343,0,0,1/4,1/4,1/4\}\\
\{4,3,2\} & \{0,1/192,-1/13824,0,0,0,0,1/4,1/4,1/4\}
\end{tabular}
\end{center}


\begin{thebibliography}{61}

\expandafter\ifx\csname natexlab\endcsname\relax\def\natexlab#1{#1}\fi
\expandafter\ifx\csname bibnamefont\endcsname\relax
  \def\bibnamefont#1{#1}\fi
\expandafter\ifx\csname bibfnamefont\endcsname\relax
  \def\bibfnamefont#1{#1}\fi
\expandafter\ifx\csname citenamefont\endcsname\relax
  \def\citenamefont#1{#1}\fi
\expandafter\ifx\csname url\endcsname\relax
  \def\url#1{\texttt{#1}}\fi
\expandafter\ifx\csname urlprefix\endcsname\relax\def\urlprefix{URL }\fi
\providecommand{\bibinfo}[2]{#2}
\providecommand{\eprint}[2][]{\url{#2}}


\bibitem[{\citenamefont{Peres}(1995)}]{Pe95}
\bibinfo{author}{\bibfnamefont{A.}~\bibnamefont{Peres}},
  \bibinfo{journal}{Phys. Lett. A} \textbf{\bibinfo{volume}{202}},
  \bibinfo{pages}{16} (\bibinfo{year}{1995}).

\bibitem[{\citenamefont{Horodecki et~al.}(2009)\citenamefont{Horodecki,
  Horodecki, Horodecki, and Horodecki}}]{HHHH09}
\bibinfo{author}{\bibfnamefont{R.}~\bibnamefont{Horodecki}},
  \bibinfo{author}{\bibfnamefont{P.}~\bibnamefont{Horodecki}},
  \bibinfo{author}{\bibfnamefont{M.}~\bibnamefont{Horodecki}},
  \bibnamefont{and}
  \bibinfo{author}{\bibfnamefont{K.}~\bibnamefont{Horodecki}},
  \bibinfo{journal}{Rev. Mod. Phys.} \textbf{\bibinfo{volume}{81}},
  \bibinfo{pages}{865} (\bibinfo{year}{2009}).

\bibitem[{\citenamefont{Amico et~al.}(2008)\citenamefont{Amico, Fazio,
  Osterloh, and Vedral}}]{amico2008entanglement}
\bibinfo{author}{\bibfnamefont{L.}~\bibnamefont{Amico}},
  \bibinfo{author}{\bibfnamefont{R.}~\bibnamefont{Fazio}},
  \bibinfo{author}{\bibfnamefont{A.}~\bibnamefont{Osterloh}}, \bibnamefont{and}
  \bibinfo{author}{\bibfnamefont{V.}~\bibnamefont{Vedral}},
  \bibinfo{journal}{Rev. Mod. Phys.} \textbf{\bibinfo{volume}{80}},
  \bibinfo{pages}{517} (\bibinfo{year}{2008}).

\bibitem[{\citenamefont{Bennett et~al.}(2000)\citenamefont{Bennett, Popescu,
  Rohrlich, Smolin, and Thapliyal}}]{bennett2000exact}
\bibinfo{author}{\bibfnamefont{C.~H.} \bibnamefont{Bennett}},
  \bibinfo{author}{\bibfnamefont{S.}~\bibnamefont{Popescu}},
  \bibinfo{author}{\bibfnamefont{D.}~\bibnamefont{Rohrlich}},
  \bibinfo{author}{\bibfnamefont{J.~A.} \bibnamefont{Smolin}},
  \bibnamefont{and} \bibinfo{author}{\bibfnamefont{A.~V.}
  \bibnamefont{Thapliyal}}, \bibinfo{journal}{Phys. Rev. A}
  \textbf{\bibinfo{volume}{63}}, \bibinfo{pages}{012307}
  (\bibinfo{year}{2000}).

\bibitem[{\citenamefont{D{\"u}r et~al.}(2000)\citenamefont{D{\"u}r, Vidal, and
  Cirac}}]{dur2000three}
\bibinfo{author}{\bibfnamefont{W.}~\bibnamefont{D{\"u}r}},
  \bibinfo{author}{\bibfnamefont{G.}~\bibnamefont{Vidal}}, \bibnamefont{and}
  \bibinfo{author}{\bibfnamefont{J.~I.} \bibnamefont{Cirac}},
  \bibinfo{journal}{Phys. Rev. A} \textbf{\bibinfo{volume}{62}},
  \bibinfo{pages}{062314} (\bibinfo{year}{2000}).

\bibitem[{\citenamefont{Verstraete et~al.}(2002)\citenamefont{Verstraete,
  Dehaene, De~Moor, and Verschelde}}]{VDMV02}
\bibinfo{author}{\bibfnamefont{F.}~\bibnamefont{Verstraete}},
  \bibinfo{author}{\bibfnamefont{J.}~\bibnamefont{Dehaene}},
  \bibinfo{author}{\bibfnamefont{B.}~\bibnamefont{De~Moor}}, \bibnamefont{and}
  \bibinfo{author}{\bibfnamefont{H.}~\bibnamefont{Verschelde}},
  \bibinfo{journal}{Phys. Rev. A} \textbf{\bibinfo{volume}{65}},
  \bibinfo{pages}{052112} (\bibinfo{year}{2002}).

\bibitem[{\citenamefont{Kauffman and Lomonaco~Jr}(2002)}]{kauffman2002quantum}
\bibinfo{author}{\bibfnamefont{L.~H.} \bibnamefont{Kauffman}} \bibnamefont{and}
  \bibinfo{author}{\bibfnamefont{S.~J.} \bibnamefont{Lomonaco~Jr}},
  \bibinfo{journal}{New J. Phys.} \textbf{\bibinfo{volume}{4}},
  \bibinfo{pages}{73} (\bibinfo{year}{2002}).

\bibitem[{\citenamefont{G.~M.~Quinta and R.~Andr{\'e}}(2018)}]{QA18}
\bibinfo{author}{\bibnamefont{G.~M.~Quinta}}
  \bibnamefont{and}
  \bibinfo{author}{~\bibnamefont{R.~Andr{\'e}}},
  \bibinfo{journal}{Phys. Rev. A} \textbf{\bibinfo{volume}{97}},
  \bibinfo{pages}{042307} (\bibinfo{year}{2018}).

\bibitem[{\citenamefont{Buniy and Kephart}(2012)}]{buniy2012algebraic}
\bibinfo{author}{\bibfnamefont{R.~V.} \bibnamefont{Buniy}} \bibnamefont{and}
  \bibinfo{author}{\bibfnamefont{T.~W.} \bibnamefont{Kephart}},
  \bibinfo{journal}{J. Phys. A: Math. Theor.} \textbf{\bibinfo{volume}{45}},
  \bibinfo{pages}{185304} (\bibinfo{year}{2012}).

\bibitem[{\citenamefont{Rao}(1947)}]{rao1947factorial}
\bibinfo{author}{\bibfnamefont{C.~R.} \bibnamefont{Rao}},
  \bibinfo{journal}{Suppl. J. Royal Stat. Soc.} \textbf{\bibinfo{volume}{9}},
  \bibinfo{pages}{128} (\bibinfo{year}{1947}).

\bibitem[{\citenamefont{Grindal et~al.}(2005)\citenamefont{Grindal, Offutt, and
  Andler}}]{GOA05}
\bibinfo{author}{\bibfnamefont{M.}~\bibnamefont{Grindal}},
  \bibinfo{author}{\bibfnamefont{J.}~\bibnamefont{Offutt}}, \bibnamefont{and}
  \bibinfo{author}{\bibfnamefont{S.~F.} \bibnamefont{Andler}},
  \bibinfo{journal}{Softw. Test. Verification Reliab.}
  \textbf{\bibinfo{volume}{15}}, \bibinfo{pages}{167} (\bibinfo{year}{2005}).

\bibitem[{\citenamefont{Hedayat et~al.}(2012)\citenamefont{Hedayat, Sloane, and
  Stufken}}]{hedayat1999orthogonal}
\bibinfo{author}{\bibfnamefont{A.~S.} \bibnamefont{Hedayat}},
  \bibinfo{author}{\bibfnamefont{N.~J.~A.} \bibnamefont{Sloane}},
  \bibnamefont{and} \bibinfo{author}{\bibfnamefont{J.}~\bibnamefont{Stufken}},
  \emph{\bibinfo{title}{Orthogonal Arrays: Theory and Applications}}
  (\bibinfo{publisher}{Springer Science \& Business Media},
  \bibinfo{year}{2012}).

\bibitem[{\citenamefont{Colbourn and Dinitz}(2006)}]{colbourn2006handbook}
\bibinfo{author}{\bibfnamefont{C.~J.} \bibnamefont{Colbourn}} \bibnamefont{and}
  \bibinfo{author}{\bibfnamefont{J.~H.} \bibnamefont{Dinitz}},
  \emph{\bibinfo{title}{Handbook of Combinatorial Designs}}
  (\bibinfo{publisher}{CRC press}, \bibinfo{year}{2006}).

\bibitem[{slo()}]{sloaneweb}
\urlprefix\url{http://neilsloane.com/oadir/}.

\bibitem[{kuh()}]{kuhfeldweb}
\urlprefix\url{http://support.sas.com/techsup/technote/ts723.html}.

\bibitem[{\citenamefont{Goyeneche and
  {\.Z}yczkowski}(2014)}]{goyeneche2014genuinely}
\bibinfo{author}{\bibfnamefont{D.}~\bibnamefont{Goyeneche}} \bibnamefont{and}
  \bibinfo{author}{\bibfnamefont{K.}~\bibnamefont{{\.Z}yczkowski}},
  \bibinfo{journal}{Phys. Rev. A} \textbf{\bibinfo{volume}{90}},
  \bibinfo{pages}{022316} (\bibinfo{year}{2014}).

\bibitem[{\citenamefont{Dean et~al.}(2015)\citenamefont{Dean, Morris, Stufken,
  and Bingham}}]{dean2015handbook}
\bibinfo{author}{\bibfnamefont{A.}~\bibnamefont{Dean}},
  \bibinfo{author}{\bibfnamefont{M.}~\bibnamefont{Morris}},
  \bibinfo{author}{\bibfnamefont{J.}~\bibnamefont{Stufken}}, \bibnamefont{and}
  \bibinfo{author}{\bibfnamefont{D.}~\bibnamefont{Bingham}},
  \emph{\bibinfo{title}{Handbook of Design and Analysis of Experiments}},
  vol.~\bibinfo{volume}{7} (\bibinfo{publisher}{CRC Press},
  \bibinfo{year}{2015}).

\bibitem[{\citenamefont{Deng and Tang}(1999)}]{deng1999generalized}
\bibinfo{author}{\bibfnamefont{L.-Y.} \bibnamefont{Deng}} \bibnamefont{and}
  \bibinfo{author}{\bibfnamefont{B.}~\bibnamefont{Tang}},
  \bibinfo{journal}{Stat. Sin.} pp. \bibinfo{pages}{1071--1082}
  (\bibinfo{year}{1999}).

\bibitem[{\citenamefont{Tang and Deng}(1999)}]{tang1999minimum}
\bibinfo{author}{\bibfnamefont{B.}~\bibnamefont{Tang}} \bibnamefont{and}
  \bibinfo{author}{\bibfnamefont{L.-Y.} \bibnamefont{Deng}},
  \bibinfo{journal}{Ann. Stat.} pp. \bibinfo{pages}{1914--1926}
  (\bibinfo{year}{1999}).

\bibitem[{\citenamefont{Tang}(2001)}]{tang2001theory}
\bibinfo{author}{\bibfnamefont{B.}~\bibnamefont{Tang}},
  \bibinfo{journal}{Biometrika} \textbf{\bibinfo{volume}{88}},
  \bibinfo{pages}{401} (\bibinfo{year}{2001}).

\bibitem[{\citenamefont{Gr{\"o}mping et~al.}(2014)\citenamefont{Gr{\"o}mping,
  Xu et~al.}}]{gromping2014generalized}
\bibinfo{author}{\bibfnamefont{U.}~\bibnamefont{Gr{\"o}mping}},
  \bibinfo{author}{\bibfnamefont{H.}~\bibnamefont{Xu}}, \bibnamefont{et~al.},
  \bibinfo{journal}{Ann. Stat.} \textbf{\bibinfo{volume}{42}},
  \bibinfo{pages}{918} (\bibinfo{year}{2014}).

\bibitem[{\citenamefont{Deng and Tang}(2002)}]{deng2002design}
\bibinfo{author}{\bibfnamefont{L.-Y.} \bibnamefont{Deng}} \bibnamefont{and}
  \bibinfo{author}{\bibfnamefont{B.}~\bibnamefont{Tang}},
  \bibinfo{journal}{Technometrics} \textbf{\bibinfo{volume}{44}},
  \bibinfo{pages}{173} (\bibinfo{year}{2002}).

\bibitem[{\citenamefont{Cheng and Ye}(2004)}]{cheng2004geometric}
\bibinfo{author}{\bibfnamefont{S.-W.} \bibnamefont{Cheng}} \bibnamefont{and}
  \bibinfo{author}{\bibfnamefont{K.~Q.} \bibnamefont{Ye}},
  \bibinfo{journal}{Annals of Statistics} pp. \bibinfo{pages}{2168--2185}
  (\bibinfo{year}{2004}).

\bibitem[{\citenamefont{Clark and Dean}(2001)}]{clark2001equivalence}
\bibinfo{author}{\bibfnamefont{J.~B.} \bibnamefont{Clark}} \bibnamefont{and}
  \bibinfo{author}{\bibfnamefont{A.~M.} \bibnamefont{Dean}},
  \bibinfo{journal}{Stat. Sin.} pp. \bibinfo{pages}{537--547}
  (\bibinfo{year}{2001}).

\bibitem[{\citenamefont{Angelopoulos et~al.}(2007)\citenamefont{Angelopoulos,
  Evangelaras, Koukouvinos, and Lappas}}]{angelopoulos2007effective}
\bibinfo{author}{\bibfnamefont{P.}~\bibnamefont{Angelopoulos}},
  \bibinfo{author}{\bibfnamefont{H.}~\bibnamefont{Evangelaras}},
  \bibinfo{author}{\bibfnamefont{C.}~\bibnamefont{Koukouvinos}},
  \bibnamefont{and} \bibinfo{author}{\bibfnamefont{E.}~\bibnamefont{Lappas}},
  \bibinfo{journal}{Metrika} \textbf{\bibinfo{volume}{66}},
  \bibinfo{pages}{139} (\bibinfo{year}{2007}).

\bibitem[{\citenamefont{Schoen et~al.}(2010)\citenamefont{Schoen, Eendebak, and
  Nguyen}}]{schoen2010complete}
\bibinfo{author}{\bibfnamefont{E.~D.} \bibnamefont{Schoen}},
  \bibinfo{author}{\bibfnamefont{P.~T.} \bibnamefont{Eendebak}},
  \bibnamefont{and} \bibinfo{author}{\bibfnamefont{M.~V.}
  \bibnamefont{Nguyen}}, \bibinfo{journal}{J. Comb. Des.}
  \textbf{\bibinfo{volume}{18}}, \bibinfo{pages}{123} (\bibinfo{year}{2010}).

\bibitem[{\citenamefont{Sun et~al.}(2008)\citenamefont{Sun, Li, and
  Ye}}]{sun2008algorithmic}
\bibinfo{author}{\bibfnamefont{D.}~\bibnamefont{Sun}},
  \bibinfo{author}{\bibfnamefont{W.}~\bibnamefont{Li}}, \bibnamefont{and}
  \bibinfo{author}{\bibfnamefont{K.}~\bibnamefont{Ye}}, \bibinfo{journal}{Stat.
  Appl.} \textbf{\bibinfo{volume}{6}}, \bibinfo{pages}{123}
  (\bibinfo{year}{2008}).

\bibitem[{\citenamefont{Tsai et~al.}(2006)\citenamefont{Tsai, Ye, and
  Li}}]{tsai2006complete}
\bibinfo{author}{\bibfnamefont{K.-J.} \bibnamefont{Tsai}},
  \bibinfo{author}{\bibfnamefont{K.}~\bibnamefont{Ye}}, \bibnamefont{and}
  \bibinfo{author}{\bibfnamefont{W.}~\bibnamefont{Li}} (\bibinfo{year}{2006}),
  \bibinfo{note}{unpublished}.

\bibitem[{\citenamefont{Stufken and Tang}(2007)}]{stufken2007complete}
\bibinfo{author}{\bibfnamefont{J.}~\bibnamefont{Stufken}} \bibnamefont{and}
  \bibinfo{author}{\bibfnamefont{B.}~\bibnamefont{Tang}},
  \bibinfo{journal}{Ann. Stat.} pp. \bibinfo{pages}{793--814}
  (\bibinfo{year}{2007}).

\bibitem[{\citenamefont{Ye et~al.}(2008)\citenamefont{Ye, Park, Li, and
  Dean}}]{Ye2008construction}
\bibinfo{author}{\bibfnamefont{K.~Q.} \bibnamefont{Ye}},
  \bibinfo{author}{\bibfnamefont{D.}~\bibnamefont{Park}},
  \bibinfo{author}{\bibfnamefont{W.}~\bibnamefont{Li}}, \bibnamefont{and}
  \bibinfo{author}{\bibfnamefont{A.~M.} \bibnamefont{Dean}},
  \bibinfo{journal}{Stat. Appl.} \textbf{\bibinfo{volume}{6}},
  \bibinfo{pages}{5} (\bibinfo{year}{2008}).

\bibitem[{\citenamefont{Pistone and Wynn}(1996)}]{pistone1996generalised}
\bibinfo{author}{\bibfnamefont{G.}~\bibnamefont{Pistone}} \bibnamefont{and}
  \bibinfo{author}{\bibfnamefont{H.~P.} \bibnamefont{Wynn}},
  \bibinfo{journal}{Biometrika} \textbf{\bibinfo{volume}{83}},
  \bibinfo{pages}{653} (\bibinfo{year}{1996}).

\bibitem[{\citenamefont{Pistone et~al.}(2000)\citenamefont{Pistone, Riccomagno,
  and Wynn}}]{pistone2000algebraic}
\bibinfo{author}{\bibfnamefont{G.}~\bibnamefont{Pistone}},
  \bibinfo{author}{\bibfnamefont{E.}~\bibnamefont{Riccomagno}},
  \bibnamefont{and} \bibinfo{author}{\bibfnamefont{H.~P.} \bibnamefont{Wynn}},
  \emph{\bibinfo{title}{Algebraic statistics: Computational Commutative Algebra
  in Statistics}} (\bibinfo{publisher}{CRC Press}, \bibinfo{year}{2000}).

\bibitem[{\citenamefont{Carlini and Pistone}(2007)}]{carlini2007hilbert}
\bibinfo{author}{\bibfnamefont{E.}~\bibnamefont{Carlini}} \bibnamefont{and}
  \bibinfo{author}{\bibfnamefont{G.}~\bibnamefont{Pistone}},
  \bibinfo{journal}{J. Stat. Theory Pract.} \textbf{\bibinfo{volume}{1}},
  \bibinfo{pages}{299} (\bibinfo{year}{2007}).

\bibitem[{\citenamefont{Goyeneche et~al.}(2016)\citenamefont{Goyeneche,
  Bielawski, and {\.Z}yczkowski}}]{goyeneche2016multipartite}
\bibinfo{author}{\bibfnamefont{D.}~\bibnamefont{Goyeneche}},
  \bibinfo{author}{\bibfnamefont{J.}~\bibnamefont{Bielawski}},
  \bibnamefont{and}
  \bibinfo{author}{\bibfnamefont{K.}~\bibnamefont{{\.Z}yczkowski}},
  \bibinfo{journal}{Phys. Rev. A} \textbf{\bibinfo{volume}{94}},
  \bibinfo{pages}{012346} (\bibinfo{year}{2016}).

\bibitem[{\citenamefont{Scott}(2004)}]{S04}
\bibinfo{author}{\bibfnamefont{A.}~\bibnamefont{Scott}},
  \bibinfo{journal}{Phys. Rev. A} \textbf{\bibinfo{volume}{69}},
  \bibinfo{pages}{052330} (\bibinfo{year}{2004}).

\bibitem[{\citenamefont{Bruns and Gubeladze}(2009)}]{bruns2010polytopes}
\bibinfo{author}{\bibfnamefont{W.}~\bibnamefont{Bruns}} \bibnamefont{and}
  \bibinfo{author}{\bibfnamefont{J.}~\bibnamefont{Gubeladze}},
  \emph{\bibinfo{title}{Polytopes, Rings, and K-theory}},
  vol.~\bibinfo{volume}{27} (\bibinfo{publisher}{Springer},
  \bibinfo{year}{2009}).

\bibitem[{\citenamefont{Hilbert}(1993)}]{hilbert1993theory}
\bibinfo{author}{\bibfnamefont{D.}~\bibnamefont{Hilbert}},
  \emph{\bibinfo{title}{Theory of Algebraic Invariants}}
  (\bibinfo{publisher}{Cambridge University Press}, \bibinfo{year}{1993}).

\bibitem[{\citenamefont{Weismantel}(1996)}]{weismantel1996hilbert}
\bibinfo{author}{\bibfnamefont{R.}~\bibnamefont{Weismantel}},
  \bibinfo{journal}{Math. Oper. Res.} \textbf{\bibinfo{volume}{21}},
  \bibinfo{pages}{886} (\bibinfo{year}{1996}).

\bibitem[{\citenamefont{Graver}(1975)}]{graver1975foundations}
\bibinfo{author}{\bibfnamefont{J.~E.} \bibnamefont{Graver}},
  \bibinfo{journal}{Math. Program.} \textbf{\bibinfo{volume}{9}},
  \bibinfo{pages}{207} (\bibinfo{year}{1975}).

\bibitem[{\citenamefont{Cox et~al.}(2007)\citenamefont{Cox, Little, and
  O'shea}}]{cox1992ideals}
\bibinfo{author}{\bibfnamefont{D.}~\bibnamefont{Cox}},
  \bibinfo{author}{\bibfnamefont{J.}~\bibnamefont{Little}}, \bibnamefont{and}
  \bibinfo{author}{\bibfnamefont{D.}~\bibnamefont{O'shea}},
  \emph{\bibinfo{title}{Ideals, Varieties, and Algorithms}},
  vol.~\bibinfo{volume}{3} (\bibinfo{publisher}{Springer},
  \bibinfo{year}{2007}).

\bibitem[{\citenamefont{Bruns and Ichim}(2010)}]{bruns2010normaliz}
\bibinfo{author}{\bibfnamefont{W.}~\bibnamefont{Bruns}} \bibnamefont{and}
  \bibinfo{author}{\bibfnamefont{B.}~\bibnamefont{Ichim}}, \bibinfo{journal}{J.
  Algebra} \textbf{\bibinfo{volume}{324}}, \bibinfo{pages}{1098}
  (\bibinfo{year}{2010}).

\bibitem[{\citenamefont{Durand et~al.}(1999)\citenamefont{Durand, Hermann, and
  Juban}}]{durand1999complexity}
\bibinfo{author}{\bibfnamefont{A.}~\bibnamefont{Durand}},
  \bibinfo{author}{\bibfnamefont{M.}~\bibnamefont{Hermann}}, \bibnamefont{and}
  \bibinfo{author}{\bibfnamefont{L.}~\bibnamefont{Juban}}, in
  \emph{\bibinfo{booktitle}{International Symposium on Mathematical Foundations
  of Computer Science}} (\bibinfo{organization}{Springer},
  \bibinfo{year}{1999}), pp. \bibinfo{pages}{92--102}.

\bibitem[{HBo()}]{HBonline}
\urlprefix\url{http://chaos.if.uj.edu.pl/~karol/Maestro7/data1.html}.

\bibitem[{\citenamefont{Euler}(1782)}]{euler1782recherches}
\bibinfo{author}{\bibfnamefont{L.}~\bibnamefont{Euler}},
  \emph{\bibinfo{title}{Recherches sur une nouvelle espece de quarres
  magiques}} (\bibinfo{publisher}{Zeeuwsch Genootschao}, \bibinfo{year}{1782}).

\bibitem[{\citenamefont{Stanley}(2007)}]{stanley2007combinatorics}
\bibinfo{author}{\bibfnamefont{R.~P.} \bibnamefont{Stanley}},
  \emph{\bibinfo{title}{Combinatorics and Commutative Algebra}},
  vol.~\bibinfo{volume}{41} (\bibinfo{publisher}{Springer Science \& Business
  Media}, \bibinfo{year}{2007}).

\bibitem[{\citenamefont{Enr{\'{\i}}quez
  et~al.}(2016)\citenamefont{Enr{\'{\i}}quez, Wintrowicz, and
  {\.Z}yczkowski}}]{EWZ16}
\bibinfo{author}{\bibfnamefont{M.}~\bibnamefont{Enr{\'{\i}}quez}},
  \bibinfo{author}{\bibfnamefont{I.}~\bibnamefont{Wintrowicz}},
  \bibnamefont{and}
  \bibinfo{author}{\bibfnamefont{K.}~\bibnamefont{{\.Z}yczkowski}},
  \bibinfo{journal}{J. Phys. Confer. Ser.} \textbf{\bibinfo{volume}{698}},
  \bibinfo{pages}{012003} (\bibinfo{year}{2016}).

\bibitem[{\citenamefont{Bostr{\"o}m and Felbinger}(2002)}]{BF02}
\bibinfo{author}{\bibfnamefont{K.}~\bibnamefont{Bostr{\"o}m}} \bibnamefont{and}
  \bibinfo{author}{\bibfnamefont{T.}~\bibnamefont{Felbinger}},
  \bibinfo{journal}{Phys. Rev. Lett.} \textbf{\bibinfo{volume}{89}},
  \bibinfo{pages}{187902} (\bibinfo{year}{2002}).

\bibitem[{\citenamefont{Gisin et~al.}(2002)\citenamefont{Gisin, Ribordy,
  Tittel, and Zbinden}}]{GRT02}
\bibinfo{author}{\bibfnamefont{N.}~\bibnamefont{Gisin}},
  \bibinfo{author}{\bibfnamefont{G.}~\bibnamefont{Ribordy}},
  \bibinfo{author}{\bibfnamefont{W.}~\bibnamefont{Tittel}}, \bibnamefont{and}
  \bibinfo{author}{\bibfnamefont{H.}~\bibnamefont{Zbinden}},
  \bibinfo{journal}{Rev. Mod. Phys.} \textbf{\bibinfo{volume}{74}},
  \bibinfo{pages}{145} (\bibinfo{year}{2002}).

\bibitem[{\citenamefont{Ac{\'\i}n and Masanes}(2016)}]{AM16}
\bibinfo{author}{\bibfnamefont{A.}~\bibnamefont{Ac{\'\i}n}} \bibnamefont{and}
  \bibinfo{author}{\bibfnamefont{L.}~\bibnamefont{Masanes}},
  \bibinfo{journal}{Nature} \textbf{\bibinfo{volume}{540}},
  \bibinfo{pages}{213} (\bibinfo{year}{2016}).

\bibitem[{\citenamefont{Coffman et~al.}(2000)\citenamefont{Coffman, Kundu, and
  Wootters}}]{CKW00}
\bibinfo{author}{\bibfnamefont{V.}~\bibnamefont{Coffman}},
  \bibinfo{author}{\bibfnamefont{J.}~\bibnamefont{Kundu}}, \bibnamefont{and}
  \bibinfo{author}{\bibfnamefont{W.~K.} \bibnamefont{Wootters}},
  \bibinfo{journal}{Phys. Rev. A} \textbf{\bibinfo{volume}{61}},
  \bibinfo{pages}{052306} (\bibinfo{year}{2000}).

\bibitem[{\citenamefont{Sudbery}(2001)}]{sudbery2001local}
\bibinfo{author}{\bibfnamefont{A.}~\bibnamefont{Sudbery}}, \bibinfo{journal}{J.
  Phys. A: Math. Gen.} \textbf{\bibinfo{volume}{34}}, \bibinfo{pages}{643}
  (\bibinfo{year}{2001}).

\bibitem[{\citenamefont{Bulutoglu and
  Margot}(2008)}]{bulutoglu2008classification}
\bibinfo{author}{\bibfnamefont{D.~A.} \bibnamefont{Bulutoglu}}
  \bibnamefont{and} \bibinfo{author}{\bibfnamefont{F.}~\bibnamefont{Margot}},
  \bibinfo{journal}{J. Stat. Plan. Inference} \textbf{\bibinfo{volume}{138}},
  \bibinfo{pages}{654} (\bibinfo{year}{2008}).

\bibitem[{\citenamefont{McKay and Piperno}(2014)}]{mckay2014practical}
\bibinfo{author}{\bibfnamefont{B.~D.} \bibnamefont{McKay}} \bibnamefont{and}
  \bibinfo{author}{\bibfnamefont{A.}~\bibnamefont{Piperno}},
  \bibinfo{journal}{J. Symb. Comput.} \textbf{\bibinfo{volume}{60}},
  \bibinfo{pages}{94} (\bibinfo{year}{2014}).

\bibitem[{\citenamefont{Luque and Thibon}(2003)}]{LT03}
\bibinfo{author}{\bibfnamefont{J.-G.} \bibnamefont{Luque}} \bibnamefont{and}
  \bibinfo{author}{\bibfnamefont{J.-Y.} \bibnamefont{Thibon}},
  \bibinfo{journal}{Phys. Rev. A} \textbf{\bibinfo{volume}{67}},
  \bibinfo{pages}{042303} (\bibinfo{year}{2003}).

\bibitem[{\citenamefont{Gelfand et~al.}(2008)\citenamefont{Gelfand, Kapranov,
  and Zelevinsky}}]{GKZ94}
\bibinfo{author}{\bibfnamefont{I.~M.} \bibnamefont{Gelfand}},
  \bibinfo{author}{\bibfnamefont{M.}~\bibnamefont{Kapranov}}, \bibnamefont{and}
  \bibinfo{author}{\bibfnamefont{A.}~\bibnamefont{Zelevinsky}},
  \emph{\bibinfo{title}{Discriminants, Resultants, and Multidimensional
  Determinants}} (\bibinfo{publisher}{Springer Science \& Business Media},
  \bibinfo{year}{2008}).

\bibitem[{\citenamefont{Hill and Wootters}(1997)}]{HW97}
\bibinfo{author}{\bibfnamefont{S.}~\bibnamefont{Hill}} \bibnamefont{and}
  \bibinfo{author}{\bibfnamefont{W.~K.} \bibnamefont{Wootters}},
  \bibinfo{journal}{Phys. Rev. Lett.} \textbf{\bibinfo{volume}{78}},
  \bibinfo{pages}{5022} (\bibinfo{year}{1997}).

\bibitem[{\citenamefont{Miyake and Wadati}(2002)}]{MW02}
\bibinfo{author}{\bibfnamefont{A.}~\bibnamefont{Miyake}} \bibnamefont{and}
  \bibinfo{author}{\bibfnamefont{M.}~\bibnamefont{Wadati}},
  \bibinfo{journal}{Quantum Inf. Comput.} \textbf{\bibinfo{volume}{2}},
  \bibinfo{pages}{540} (\bibinfo{year}{2002}).

\bibitem[{\citenamefont{Gour and Wallach}(2010)}]{GW10}
\bibinfo{author}{\bibfnamefont{G.}~\bibnamefont{Gour}} \bibnamefont{and}
  \bibinfo{author}{\bibfnamefont{N.~R.} \bibnamefont{Wallach}},
  \bibinfo{journal}{J. Math. Phys.} \textbf{\bibinfo{volume}{51}},
  \bibinfo{pages}{112201} (\bibinfo{year}{2010}).


\bibitem[{\citenamefont{Goyeneche et~al.}(2017)\citenamefont{Goyeneche, Raissi,
  Di~Martino, and {\.Z}yczkowski}}]{goyeneche2017entanglement}
\bibinfo{author}{\bibfnamefont{D.}~\bibnamefont{Goyeneche}},
  \bibinfo{author}{\bibfnamefont{Z.}~\bibnamefont{Raissi}},
  \bibinfo{author}{\bibfnamefont{S.}~\bibnamefont{Di~Martino}},
  \bibnamefont{and}
  \bibinfo{author}{\bibfnamefont{K.}~\bibnamefont{{\.Z}yczkowski}},
  \bibinfo{journal}{Phys. Rev. A}  \textbf{\bibinfo{volume}{97}},
  \bibinfo{pages}{062326} (\bibinfo{year}{2018}).

\bibitem[{\citenamefont{Holweck et~al.}(2017)\citenamefont{Holweck, Luque, and
  Thibon}}]{HLT17}
\bibinfo{author}{\bibfnamefont{F.}~\bibnamefont{Holweck}},
  \bibinfo{author}{\bibfnamefont{J.-G.} \bibnamefont{Luque}}, \bibnamefont{and}
  \bibinfo{author}{\bibfnamefont{J.-Y.} \bibnamefont{Thibon}},
  \bibinfo{journal}{J. Math. Phys.} \textbf{\bibinfo{volume}{58}},
  \bibinfo{pages}{022201} (\bibinfo{year}{2017}).

\bibitem[{\citenamefont{Olver}(1999)}]{P99}
\bibinfo{author}{\bibfnamefont{P.~J.} \bibnamefont{Olver}},
  \emph{\bibinfo{title}{Classical Invariant Theory}}, vol.~\bibinfo{volume}{44}
  (\bibinfo{publisher}{Cambridge University Press}, \bibinfo{year}{1999}).

\end{thebibliography}
\end{document}